\documentclass[11pt,reqno]{amsart} 
\usepackage{amsfonts, amsmath, amssymb, amsthm, color}
\usepackage[margin=0.97in]{geometry}
\usepackage{refcheck}
\usepackage{esint}

\newtheorem{thm}{Theorem}[section]
\newtheorem{cor}[thm]{Corollary}
\newtheorem{prop}[thm]{Proposition}
\newtheorem{lemma}[thm]{Lemma}

\theoremstyle{definition}

\newtheorem{remark}[thm]{Remark}

\newcommand{\dist}{\textnormal{dist}}

\allowdisplaybreaks
\numberwithin{equation}{section}

\makeatletter
\@namedef{subjclassname@2020}{%
	\textup{2020} Mathematics Subject Classification}
\makeatother

\begin{document}
	\title[Existence results of Singular Toda systems with sign-changing weight functions]{Existence results of Singular Toda systems with sign-changing weight functions}
	\author{Qiang Fei}
	\address{School of Mathematics and Statistics, Central South University,
		Changsha 410083, Hunan, Peoples Republic of China}
	\email{math\_qiangfei@163.com}
	
	\maketitle
	\section*{Abstract}
	We consider the existence problem of the following Singular Toda system on a compact Riemann surface $(\Sigma, g)$ without boundary
	\begin{equation*}
		\begin{cases}
			-\Delta_gu_1=2\overline{\rho}_1\Big({\frac{h_1e^{u_1}}{\int_{\Sigma}h_1e^{u_1}dV_g}}-1\Big)-\rho_2\Big({\frac{h_2e^{u_2}}{\int_{\Sigma}h_2e^{u_2}dV_g}}-1\Big)-4\pi\alpha_1(\delta_0-1),\\
			\\
			-\Delta_gu_2=2\rho_2\big({\frac{h_2e^{u_2}}{\int_{\Sigma}h_2e^{u_2}dV_g}}-1\big)-\overline{\rho}_1\big({\frac{h_1e^{u_1}}{\int_{\Sigma}h_1e^{u_1}dV_g}}-1\big)-4\pi\alpha_2(\delta_0-1),
		\end{cases}
	\end{equation*}
	where $h_1,\,h_2$ are sign-changing smooth functions,\ $\overline{\rho}_1:=4\pi(1+\overline{\alpha}_1),\,0<\rho_2<4\pi(1+\overline{\alpha}_2),\,\overline{\alpha}_i=\min\{0,\alpha_i\},\,\alpha_i>-1,\,i=1,2$. By relying on the proof framework established in\,\cite{DJLW}, the Pohozaev identity and the classical blow-up analysis, we prove the existence theorem under some appropriate condition. Our results generalize Jost-Wang's results\,\cite{JLW} from regular Toda system with positive functions to the singular Toda system involving sign-changing weight functions.
	\par
	
	\section{Introduction}
	The following $\rm{SU}(n+1)$ Toda system
	\begin{equation}\label{G-SU}
		-\Delta u_i=\sum\limits^N_{j=1} a_{ij}e^{u_j},\ \ x\in\Sigma,\ i=1,\cdots,N,
	\end{equation}
	where $\Delta$ represents Laplacian operator and $\rm{A}_n=(a_{ij})_{n\times n}$ the Cartan matrix 
	\begin{equation*}
		A_n=
		\begin{pmatrix}
			2&-1&0&0&\cdots&0\\
			-1&2&-1&0&\cdots&0\\
			0&-1&2&-1&\cdots&0\\
			\vdots&\vdots&\vdots&\vdots&\ \ &\vdots&\\
			0&\cdots&\cdots&-1&2&-1\\
			0&\cdots&\cdots&0&-1&2
		\end{pmatrix},
	\end{equation*}
	appears in many models in different disciplines of science, such as conformal geometry and mathematical physics. In geometry, it palys an important role in the description of holomorphic curves in $\mathbb{CP}^n$ and also has a close connection with the flat $\rm{SU}(n+1)$ connection, complete integrablity and harmonic sequences, see e.g.\cite{BJRW,B-W,Calabi,Chern-W,Doliwa,Guest,L-S} for references.
	In mathematical physics, it arises in models for non-abelian Chern-Simons vortices, which might be applied in high-temperature superconductivity and which also appear in a much wider range of variety than the Yang-Mills framework, see e.g.\cite{Tarantello-A},\,\cite{Tarantello-S} and\,\cite{Yang-S} for details.
	\par

	As $n=1$, system \eqref{G-SU} is reduced to a single Liouville equation, which has been extensively studied for decades. Consider the existence problem of the following Liouville type equation
	\begin{equation}\label{KZ}
		-\Delta u=8{\pi}he^{u}-8\pi,
	\end{equation}
	which is called Kazdan-Warner problem in 1974 and has been engaged in\,\cite{DJLW} successfully. In this article, Ding-Jost-Li-Wang proved the existence result of \eqref{KZ} under some proper geometric condition.  
	\par

	More specifically, in order to obtain the solvability of \eqref{KZ}, Ding-Jost-Li-Wang considered the perturbed equation and corresponding perturbation energy functional $I_{8\pi-\epsilon}$. By exploiting variational method, the minimum of $I_{8\pi-\epsilon}$ could be attained at some $u_{\epsilon}$. Then by resorting to the blow up analysis, they acquired that there at most exists one blow-up point for $u_{\epsilon}$. Subsequently, if the minimizing sequence $u_{\epsilon}$ didn't converge in $H^1(\Sigma)$, they derived the lower boundedness of the limit energy functional $I_{8\pi-\epsilon}(u_{\epsilon})$ as $\epsilon\rightarrow0$. Finally, by constructing a proper bubbling sequence $\phi_{\epsilon}$ and imposing some proper geometric condition, they obtained the limit value of the perturbed energy functional $I_{8\pi-\epsilon}(\phi_{\epsilon})$ is less than this lower boundedness as $\epsilon\rightarrow0$, which makes a contradiction. This fact implies that $u_{\epsilon}$ will converge to the minimum point of $I_{8\pi}$, which coincides with the solution of \eqref{KZ}. Based on the method of Ding et al.\cite{DJLW}, Zhu-Yang \cite{YangZhu,Zhu-a-s, Zhu-an},\ Sun-Zhu \cite{SunZhu-E} generalized the existence result of positive weight function to the case with nonnegative or sign-changing function. Later, similar existence results could also be acquired by flow approach, see \cite{LiXu,LiZhu,SunZhu-G,WY} and the references therein. In light of the above results about scalar Liouville equation, it is natural to generalize the existence problem to the system case.
	\par


	In this article, we focus on the existence problem of following Singular Toda system on a compact Riemann surface without boundary,
	\begin{equation}\label{main-eq}
		\begin{cases}
			-\Delta_gu_1=2\overline{\rho}_1\Big({\frac{h_1e^{u_1}}{\int_{\Sigma}h_1e^{u_1}dV_g}}-1\Big)-\rho_2\Big({\frac{h_2e^{u_2}}{\int_{\Sigma}h_2e^{u_2}dV_g}}-1\Big)-4\pi\alpha_1(\delta_0-1),\\
			\\
			-\Delta_gu_2=2\rho_2\big({\frac{h_2e^{u_2}}{\int_{\Sigma}h_2e^{u_2}dV_g}}-1\big)-\overline{\rho}_1\big({\frac{h_1e^{u_1}}{\int_{\Sigma}h_1e^{u_1}dV_g}}-1\big)-4\pi\alpha_2(\delta_0-1),
		\end{cases}
	\end{equation}
	where $h_1,\,h_2$ are smooth sign-changing functions, $\overline{\rho}_1:=4\pi(1+\overline{\alpha}_1),\,0<\rho_2<4\pi(1+\overline{\alpha}_2),\,\overline{\alpha}_i:=\min\{\alpha_i,\,0\}$ and $\alpha_i>-1,\,i=1,2$. For simplicity, we can define the volume of $\Sigma$ as $|\Sigma|_g=1$.
	\par

	To illustrate the core attributes of our problem, we firstly de-singularize \eqref{main-eq} via a simple change of variables. Consider the Green function $G_p(x)$ satisfying
	\begin{equation*}
		\begin{cases}
			-\Delta_gG_p(x)=\delta_p-{\frac{1}{|\Sigma|_g}}\ \ \text{in}\ \Sigma,\\
			\\
			\int_{\Sigma}G_p(x)dV_g(x)=0.
		\end{cases}
	\end{equation*}
	\par
	
	In a conformal coordinate system centered at $p$, we can postulate that $G_p(x)={\frac{1}{8\pi}}(-4\log r+A(p)+\sigma(x))$, where $r:=\dist_g(x,p)$ and $\sigma(x)=O(r)$ as $r\rightarrow0$. Then by the substitution
	\begin{equation*}
		v_i(x):=u_i+4\pi\alpha_iG_0-\log\int_{\Sigma}h_ie^{u_i}dV_g,\ \ H_i(x):=4\pi\alpha_iG_0,\ i=1,2,
	\end{equation*}
	problem \eqref{main-eq} can be translated into the following system
	\begin{equation}\label{main-eq-1*}
		\begin{cases}
			-\Delta_gv_1=2\overline{\rho}_1\Big(h_1e^{-H_1}e^{v_1}-1\Big)-\rho_2\Big({h_2e^{-H_2}e^{v_2}}-1\Big),\\
			\\
			-\Delta_gv_2=2\rho_2\Big(h_2e^{-H_2}e^{v_2}-1\Big)-\overline{\rho}_1\Big(h_1e^{-H_1}e^{v_1}-1\Big),
		\end{cases}	
	\end{equation}
	where $v_1,\,v_2\in\mathcal{H}$ and $\mathcal{H}$ is defined as the following form,
	\begin{equation*}
		\mathcal{H}=\left\{(v_1,\,v_2)\in H^1(\Sigma)\times H^1(\Sigma):\ \int_{\Sigma}h_1e^{-H_1}e^{v_1}dV_g=1,\,\,\int_{\Sigma}h_2e^{-H_2}e^{v_2}dV_g=1\right\}.
	\end{equation*}
	\par

	The energy functional $J_{\overline{\rho}_1}$ corresponding to\,\eqref{main-eq-1*} is as below:
	\begin{equation*}
		J_{\overline{\rho}_1}(v_1,\,v_2)=\int_{\Sigma}Q(v_1,v_2)dV_g-\overline{\rho}_1\log\int_{\Sigma}h_1e^{-H_1}e^{v_1-\overline{v}_1}dV_g-\rho_2\log\int_{\Sigma}h_2e^{-H_2}e^{v_2-\overline{v}_2}dV_g,
	\end{equation*}
	where $Q(v_1,v_2)={\frac{1}{3}}\big(|\nabla_gv_1|^2+|\nabla_gv_2|^2+\nabla_gv_1\nabla_gv_2\big)$. Refer to the following Moser-Trudinger inequality with singularities \cite{BM},
	\begin{equation}\label{MT}
		4\pi\big(1+\min\{\overline{\alpha}_1,\overline{\alpha}_2\}\big)\sum^2\limits_{i=1}\log\int_{\Sigma}e^{-H_i}e^{v_i-\overline{v}_i}dV_g\leq\int_{\Sigma}Q(v_1,\,v_2)dV_g+C_{\Sigma},\ \ \forall v_1,\,v_2\in H^1(\Sigma).
	\end{equation}
	\par

	We can take into account the following perturbed energy functional $J_{\rho^n_1}$ with $0<\rho_2<4\pi(1+\overline{\alpha}_2)$ and $\rho^n_1$ strictly increasingly converges to $\overline{\rho}_1:=4\pi(1+\overline{\alpha}_1)$, i.e., $\rho^n_1\uparrow\overline{\rho}_1$ strictly. Due to this Moser-Trudinger inequality \eqref{MT}, it is clear that $J_{\rho^n_1}$ is coercive in the set $\mathcal{H}$, which means that $J_{\rho^n_1}$ attains its infimum at some $\big(v^n_1,\,v^n_2\big)\in\mathcal{H}$.
	\par
	
	Thus, in order to attack the existent problem for system \eqref{main-eq-1*}, we can consider the following perturbed system corresponding to the energy functional $J_{\rho^n_1}(v^n_1,\,v^n_2)$,
	\begin{equation}\label{main-eq-1}
		\begin{cases}
			-\Delta_gv^n_1=2\rho^n_1\Big(h_1e^{-H_1}e^{v^n_1}-1\Big)-\rho_2\Big(h_2e^{-H_2}e^{v^n_2}-1\Big),\\
			\\
			-\Delta_gv^n_2=2\rho_2\Big(h_2e^{-H_2}e^{v^n_2}-1\Big)-\rho^n_1\Big(h_1e^{-H_1}e^{v^n_1}-1\Big),
		\end{cases}	
	\end{equation}
	where $\rho^n_1\uparrow\overline{\rho}_1:=4\pi(1+\overline{\alpha}_1)$ strictly,\ $0<\rho_2<4\pi(1+\overline{\alpha}_2),\,h_1,\,h_2$ are the smooth sign-changing functions and $(v^n_1,\,v^n_2)\in\mathcal{H}$ are the solutions of \eqref{main-eq-1}.
	\par

	In \cite{JLW}, under some geometric condition and $h_1,\,h_2$ being positive weight functions, Jost-Lin-Wang obtained the existence result of \eqref{main-eq} without singular sources for $\rho_1=4\pi,\,0<\rho_2<4\pi$ and $\rho_1=\rho_2=4\pi$. Roughly speaking, relying on classical blow-up analysis and the Pohozaev identity\,\cite[Proposition 2.8]{JW-A} deduced by potential estimates, they classified the blow-up value of the SU(3)Toda system, which could be only one of $(4\pi,\,0),\ (0,\,4\pi),\ (4\pi,\,8\pi),\ (8\pi,\,4\pi)$ and $(8\pi,\,8\pi)$. More specifically, they revealed that the scope of blow-up value is within the range of $[4\pi,\,8\pi]$ after discussing the blow-up scenario, which could yield the exact value by utilizing the Pohozaev identity again. Consequently, the possible blow-up value on Riemann surface could be only $(4\pi,0)$ or $(0,4\pi)$. Then by adding some geometric condition deduced from the sharp estimates in\,\cite{Chen-Lin-S-M}, Jost-Lin-Wang obtained the compactness of the minimizing sequence corresponding to the perturbed functional. This fact led to the existence result of \eqref{main-eq} without singular sources for both critical and partial critical parameters. 
	\par

	Following that, relying on the selection process, Pohozaev identity and the discussion on the interaction of bubbling solutions, Lin-Wei-Zhang in\,\cite{LWZ} demonstrated that the limit of energy concentration of singular SU(3)Toda systems constituted a finite set. Furthermore, they also established a uniform estimate for fully bubbling solutions by employing Pohozaev identity. Later, Jevnikar-Wei-Yang\,\cite{JWY} obtained similar result by considering the elliptic sinh-Gordon equation. In\,\cite{WWZ},\,\cite{WZ}, the authors also considered the bubbling phenomenon of\,\eqref{main-eq} without singularities when both parameters tend to the critical values. For more classification result of singular Toda system with critical parameters, see \cite{Chen-Lin-O, Chen-Lin-S, LWY}, etc. 
	\par

	Compared to scalar Liouville equation, the blow-up phenomena for system are more complicated. In contrast to the system with positive weight functions, the disturbed energy caused by sign-changing functions results in a different quantitative analysis of blow-up energy. Furthermore, the presence of singularity adds more complexity to the problem of existence. These factors mentioned above make the existence problem for our case more attractive.	
	\par

	Inspired by the ideas in \cite{DJLW} and \cite{Zhu-a-s}, we can also prove the compactness of the minimizing sequence to the perturbed energy functional by imposing some appropriate conditions. Then as $\rho^n_1$ converges to the critical parameter   $\overline{\rho}_1:=4\pi(1+\overline{\alpha}_1)$, we can derive the limit of minimizing sequence $(v^n_1,\,v^n_2)$, which corresponds to a solution of the original system \eqref{main-eq}.
	\par

	Roughly speaking, firstly by applying blow-up analysis and Pohozaev identity to the perturbed system \eqref{main-eq-1}, we can deduce that $v^n_1$ will blow up at most one point and $v^n_2$ is uniformly bounded from above. In particular, $v^n_1$ can't blow up at the origin as $\alpha_1>0$ and $v^n_1$ will blow up at the origin as $\alpha_1<0$. Subsequently, by splitting the Riemann surface $(\Sigma,g)$ into three domains and calculating the perturbed functional $J_{\rho^n_1}$ on each region, we can establish the lower bound of the infimum value for the energy functional corresponding to \eqref{main-eq-1*} according to the situation $\alpha_1\geq0$ and $\alpha_1<0$. Later, by constructing an appropriate test function, for the case $\alpha_1\geq0$, we can prove the compactness of the minimizing sequence under some appropriate geometric condition. On the other hand, as $\alpha_1<0$, if the infimum value of the energy functional $J_{\overline{\rho}_1}$ is less than this lower bound, we can also obtain the compactness of the minimizing sequence. As a key point of this article, we will elaborate on the blow-up analysis strategy illuminated in Lemma \ref{Sign-wf},\,\ref{PI},\,\ref{CPoint}. However, before that, for any blow-up point $z$, we also need to define the limit energy of $v^n_i,\,i=1,2$ at $z$ as follows:
	\begin{equation}\label{EL}
		\begin{aligned}
			\sigma_1(z):=\lim\limits_{r\rightarrow0}\lim\limits_{n\rightarrow+\infty}\rho^n_1\int_{B_{r}(z)}h_1e^{-H_1}e^{v^n_1}dV_g,\ \ \ \,\sigma_2(z):=\lim\limits_{r\rightarrow0}\lim\limits_{n\rightarrow+\infty}\rho_2\int_{B_{r}(z)}h_2e^{-H_2}e^{v^n_2}dV_g.
		\end{aligned}
	\end{equation}
	\par

	At first, by relying on Brezis-Merle's result \cite[Theorem 1]{Brezis-M}, we obtain the positive value of the sign-changing weight functions at every common blow-up point, which eliminates the possibility of non-positive disturbances. In addition, when considering the blow-up situation of one component of the minimizing sequence $(v^n_1,v^n_2)$, we can also rule out the possibility of the weight function being non-positive near the corresponding blow-up point, according to the strong maximum principle. Subsequently, by exploiting the Pohozaev identity, we obtain that the local mass of $v^n_1$ at the corresponding blow-up point is no less than the total mass. In the meanwhile, by utilizing the boundedness of oscillation on the boundary, we can also deduce the vanishing of the mass of $v^n_1$ in the remaining region. Thus, we obtain that $v^n_1$ blows up at most one point. In particular, as $\alpha_1>0,\ v^n_1$ can not blow up at the origin\,; as $\alpha_1<0,\ v^n_1$ will blow up at the origin. Then, by applying the Pohozaev identity again and the constraint on the parameter $\rho_2$, we can also infer that $v^n_2$ is uniformly bounded from above. The following are our main theorems.
	\par

	\begin{thm}\label{Exist-thm}
		Let $(\Sigma,g)$ be a compact Riemann surface without boundary, $K(x)$ be the Gaussian curvature of $(\Sigma,g)$ and $h_1,\,h_2$ be the sign-changing smooth functions. Define $\Sigma_+:=\{x\in\Sigma:\ h_1(x)>0\}$.
		\par
		
		\textbf{(i)} If $\alpha_1\geq0$, then under the assumption that 
		\begin{equation*}
		\Delta_g\log h_1(x)+8\pi-4\pi\alpha_1-\rho_2-2K(x)>0,\ \forall x\in\Sigma_+,
		\end{equation*}
there exists a solution of \eqref{main-eq}.
\par

\textbf{(ii)} If $\alpha_1<0$, then under the assumption that
\begin{equation*}
	\begin{aligned}
		\inf\limits_{H^1(\Sigma)}J_{\overline{\rho}_1}<&\int_{\Sigma}|\nabla_gw_0|^2dV_g-\overline{\rho}_1\int_{\Sigma}\nabla_gG_{0}\nabla_gw_0dV_g-\rho_2\log\int_{\Sigma}h_2e^{-4\pi\alpha_2G_0-\overline{\rho}_1G_{0}}e^{2w_0}dV_g\\
		&\quad-{\frac{(\overline{\rho}_1)^2}{8\pi}} A(0)+\overline{\rho}_1w_0(0)-\overline{\rho}_1\log\big({\frac{ h_1(0)e^{-\frac{\overline{\alpha}_1}{2}A(0)}}{1+\overline{\alpha}_1}}\big)-\overline{\rho}_1\log\pi-\overline{\rho}_1, 
	\end{aligned}
\end{equation*}
there exists a solution of \eqref{main-eq}.	
\end{thm}	
\par

As this paper is about to be completed, a related result by Sun-Zhu\,\cite{SunZhu-x-E} appeared on arXiv. They establish conclusion (i) of Theorem \ref{Exist-thm} for the particular case where $\alpha_1=\alpha_2=0$.
\par

Moreover, following from Conclusion (ii) of Theorem \ref{Exist-thm}, we have the following corollary.			
    \begin{cor}\label{Non-thm}
		Let $(\Sigma,g)$ be a compact Riemann surface without boundary, $K(x)$ be the Gaussian curvature of $(\Sigma,g)$ and $h_1,\,h_2$ be the sign-changing smooth functions. When $\alpha_1<0$, if the minimizing sequence $(v^n_1,v^n_2)$ of the perturbed functional $J_{\rho^n_1}$ does not converge in $H^1(\Sigma)\times H^1(\Sigma)$, then
		\begin{equation*}
			\begin{aligned}
				\inf\limits_{H^1(\Sigma)}J_{\overline{\rho}_1}=&\int_{\Sigma}|\nabla_gw_0|^2dV_g-\overline{\rho}_1\int_{\Sigma}\nabla_gG_{0}\nabla_gw_0dV_g-\rho_2\log\int_{\Sigma}h_2e^{-4\pi\alpha_2G_{0}-\overline{\rho}_1G_{0}}e^{2w_0}dV_g\\
				&\quad-{\frac{(\overline{\rho}_1)^2}{8\pi}} A(0)+\overline{\rho}_1w_0(0)-\overline{\rho}_1\log\big({\frac{ h_1(0)e^{-\frac{\overline{\alpha}_1}{2}A(0)}}{1+\overline{\alpha}_1}}\big)-\overline{\rho}_1\log\pi-\overline{\rho}_1.
			\end{aligned}
		\end{equation*}
	 \end{cor}
	\par
	
	\begin{remark}
		In comparison to \cite{JLW}, we generalize the geometric condition to the singular case, which involves sign-changing functions.
	\end{remark}		
	
	The paper is organized as follows: in Section 2, by exploiting the classical blow-up analysis and the Pohozaev identity, we confirm that $v^n_1$ will blow up at most one point and $v^n_2$ is uniformly bounded from above. In particular, as $\alpha_1>0,\ v^n_1$ can not blow up at the origin; as $\alpha<0,\ v^n_1$ will blow up at the origin. In Section 3, according to the situation $\alpha_1\geq0$ and $\alpha_1<0$, through straightforward computations based on the blow-up behavior of $v^n_1$ and the interaction between $v^n_1$ and $v^n_2$, we derive the lower bound of the infimum value of the energy functional corresponding to \eqref{main-eq-1*} separately. In Section 4, under some appropriate assumption, we complete the Theorem \ref{Exist-thm} by constructing an appropriate test function.
	\par

	\section{Preliminary Knowledge}

	\begin{lemma}\label{u-bdd}
		Suppose $(v^n_1,\,v^n_2)\in\mathcal{H}$ attains the infimum of $J_{\rho^n_1}$, then there exist two positive constants $c_1$ and $c_2$ such that
		\begin{equation}\label{u-bdd*}
			c_1\leq\int_{\Sigma}e^{-H_i}e^{v^n_i}dV_g\leq c_2,\ i=1,2.
		\end{equation}
	\end{lemma}
	
	\begin{proof}
		Since $(v^n_1,\,v^n_2)\in\mathcal{H}$, we have that
		\begin{equation}\label{u-bdd-1}
			\begin{aligned}
				{\frac{1}{\max\limits_{x\in\Sigma}h_i(x)}}={\frac{\int_{\Sigma}h_ie^{-H_i}e^{v^n_i}dV_g}{\max\limits_{x\in\Sigma}h_i(x)}}\leq\int_{\Sigma}e^{-H_i}e^{v^n_i}dV_g,\ i=1,2,
			\end{aligned}
		\end{equation}
		then the left side of inequality \eqref{u-bdd*} can be attained by taking $c_1:=\min\{{\frac{1}{\max\limits_{x\in\Sigma}h_1(x)}},\,{\frac{1}{\max\limits_{x\in\Sigma}h_2(x)}}\}$.
		
		As $n\rightarrow+\infty$, for any $\rho^n_1$ sufficiently close to $4\pi(1+\overline{\alpha}_1)$, by utilizing Moser-Trudinger inequality\,\eqref{MT}, Jensen's inequality and the boundedness of $J_{\rho^n_1}(v^n_1,v^n_2)$, we can deduce that
		\begin{equation*}
			\begin{aligned}
				&\log\int_{\Sigma}e^{-H_1}e^{v^n_1}dV_g+\log\int_{\Sigma}e^{-H_2}e^{v^n_2}dV_g,\\
				\leq&{\frac{1}{\rho_0}}\Big(J_{\rho^n_1}(v^n_1,v^n_2)+\rho^n_1\log\int_{\Sigma}h_1e^{-H_1}e^{v^n_1-\overline{v}^n_1}dV_g+\rho_2\log\int_{\Sigma}h_2e^{-H_2}e^{v^n_2-\overline{v}^n_2}dV_g\Big)+\overline{v}^n_1+\overline{v}^n_2+C_{\Sigma},\\
				\leq&\big(1-{\frac{\rho^n_1}{\rho_0}}\big)\int_{\Sigma}\big(v^n_1-H_1\big)dV_g+\big(1-{\frac{\rho_2}{\rho_0}}\big)\int_{\Sigma}\big(v^n_2-H_2\big)dV_g+\tilde{C}_{\Sigma},\\
				\leq&\big(1-{\frac{\rho^n_1}{\rho_0}}\big)\log\int_{\Sigma}e^{-H_1}e^{v^n_1}dV_g+\big(1-{\frac{\rho_2}{\rho_0}}\big)\log\int_{\Sigma}e^{-H_2}e^{v^n_2}dV_g+\tilde{C}_{\Sigma},
			\end{aligned}
		\end{equation*}
		where $\rho_0:=4\pi\big(1+\min\{\overline{\alpha}_1,\overline{\alpha}_2\}\big)$. After simplification, it yields that
		\begin{equation}\label{u-bdd-3}
			{\frac{\rho^n_1}{\rho_0}}\log\int_{\Sigma}e^{-H_1}e^{v^n_1}dV_g+{\frac{\rho_2}{\rho_0}}\log\int_{\Sigma}e^{-H_2}e^{v^n_2}dV_g\leq\tilde{C}_{\Sigma}.
		\end{equation}
		\par

		Combining \eqref{u-bdd-1} and\,\eqref{u-bdd-3} together, it indicates that
		\begin{equation*}
			\int_{\Sigma}e^{-H_1}e^{v^n_1}dV_g\leq e^{2\tilde{C}_{\Sigma}+{\frac{2\rho_2}{\rho_0}}\big|\log\max\limits_{x\in\Sigma}h_2(x)\big|},\ \int_{\Sigma}e^{-H_2}e^{v^n_2}dV_g\leq e^{\frac{4\pi(1+\overline{\alpha}_1)\big(\tilde{C}_{\Sigma}+\big|\log\max\limits_{x\in\Sigma}h_1(x)\big|\big)}{\rho_2}}.
		\end{equation*}
		\par

		Then the right side of \eqref{u-bdd*} has also been verified by choosing
		\begin{equation*}
			c_2:=\max\left\{\exp\{2\tilde{ C}_{\Sigma}+{\frac{2\rho_2}{\rho_0}}\big|\log\max\limits_{x\in\Sigma}h_2(x)\big|\},\ \exp\{\frac{4\pi(1+\overline{\alpha}_1)\big(\tilde{C}_{\Sigma}+\big|\log\max\limits_{x\in\Sigma}h_1(x)\big|\big)}{\rho_2}\}\right\}.
		\end{equation*}
	\end{proof}
	\par

	\begin{remark}

		before establishing the framework of the blow-up scenario, it is essential to confirm the fact that $h_i(z)>0$ at the corresponding blow-up point $z$ of $v^n_i,\,i=1,2$. In order to achieve this target, we can refer to \cite[Theorem 1]{Brezis-M} as follows.
	\end{remark}
	\par

	\begin{lemma}\label{Brezis-Merle}
		Let $\Omega\subset\Sigma$ be a smooth domain. Assume $u$ is a solution to
		\begin{equation*}
			\begin{cases}
				-\Delta_gu=f,\ x\in\Omega\\
				u|_{\partial\Omega}=0,
			\end{cases}
		\end{equation*}
		where $f\in L^1(\Omega),$ then for every $0<\delta<4\pi$, there is a constant $C>0$ depending only on $\delta$ and $\Omega$ such that
		\begin{equation*}
			\int_{\Omega}\exp\Big({\frac{\big(4\pi-\delta\big)|u|}{\|f\|_{L^1(\Omega)}}}\Big)dV_g\leq C.
		\end{equation*}
	\end{lemma}
	\par

	Having this lemma at our disposal, we will elaborate on how to determine the sign of the weight function at the blow-up points.
	\par

	\begin{lemma}\label{Sign-wf}
		If $v^n_i$ blow up at the point $z$, we can have that $h_i(z)>0,\,i=1,2$.
	\end{lemma}
	\begin{proof}
		\noindent\textbf{Case 1.}
		Suppose $z$ is a common blow up point of $v^n_1$ and $v^n_2$, then there exists a small enough ball $B_{4r}(z)$, such that $\overline{B}_{4r}(z)\cap\big(S_1\cup S_2\big)=\{z\}$. Translate \eqref{main-eq-1} into the following system:
		\begin{equation}\label{Sign-wf1}
			\begin{cases}
				-\Delta_g\big(v^n_1+2v^n_2\big)=3\rho_2h_2e^{-H_2}e^{v^n_2}-3\rho_2,\\
				\\
				-\Delta_g\big(v^n_2+2v^n_1\big)=3\rho^n_1h_1e^{-H_1}e^{v^n_1}-3\rho^n_1,\ x\in B_{4r}(z).
			\end{cases}
		\end{equation}
		\par
		\item{(i)}
		Suppose $h_1(z)<0$ or $h_2(z)<0$, for the sake of generality, let $h_1(z)<0$. Then applying maximum principle to \eqref{Sign-wf1}, we can obtain that
		\begin{equation*}
			v^n_2+2v^n_1\leq\big(v^n_2+2v^n_1\big)|_{\partial B_{4r}(z)}<+\infty,\ \ \text{for any}\ x\in B_{4r}(z),
		\end{equation*}
		which contradicts to the assumption that $z$ is a common blow-up point of $v^n_1$ and $v^n_2$.
		\par
		
		\item{(ii)}
		Suppose $h_1(z)=0$ or $h_2(z)=0$, without loss of generality, let $h_2(z)=0$. If $-1<\alpha_2<0$, then there exist some constant $p_0>1,\,q_0>1$ which both rely on $\alpha_2$, such that $-\alpha_2p_0q_0<1$. For any fixed $\delta\in(0,4\pi)$, due to the continuity of $h_2(x)$, we can choose $r$ sufficiently small, such that for any $x\in B_{4r}(z)$, 
		\begin{equation*}
			|h_2(x)|\leq{\frac{4\pi-2\delta}{3\rho_2c_2}}\ \ \text{as}\ \alpha_2\geq0\ \ \text{or}\,\  |h_2(x)|\leq{\frac{(4\pi-\delta)(q_0-1)}{3\rho_2c_2p_0q_0}}\ \ \text{as}\ -1<\alpha_2<0.
		\end{equation*}
		\par
		
		This implies that
		\begin{equation}\label{Sign-wf-2}
			\begin{aligned}
				\int_{B_{4r}(z)}3\rho_2|h_2|\cdot e^{-H_2}e^{v^n_2}dV_g\leq\max\limits_{x\in B_{4r}(z)}|h_2|\cdot3\rho_2\Big(\int_{\Sigma}e^{-H_2}e^{v^n_2}dV_g\Big)\leq
				\begin{cases}
					4\pi-2\delta\ \ \text{as}\ \alpha_2\geq0,\\
					\\
					{\frac{(4\pi-\delta)(q_0-1)}{p_0q_0}}\,\ \text{as}\,-1<\alpha_2<0.
				\end{cases}
			\end{aligned}
		\end{equation}
		\par

		Let $\tau^n_1$ satisfy the following equation,
		\begin{equation}\label{Sign-wf-3}
			\begin{cases}
				-\Delta_g\tau^n_1=3\rho_2h_2e^{-H_2}e^{v^n_2}\ \text{in}\ B_{4r}(z),\\
				\tau^n_1|_{\partial B_{4r}(z)}=0.
			\end{cases}
		\end{equation}
		\par
		
		By applying \eqref{Sign-wf-2} and lemma \ref{Brezis-Merle} to \eqref{Sign-wf-3}, it indicates there exists a constant $C>0$, such that
		\begin{equation*}
			\|\tau^n_1\|_{L^p(B_{4r}(z))}\leq C,\ \|e^{|\tau^n_1|}\|_{L^p(B_{4r}(z))}\leq C\  \ \text{for}\ p=
			\begin{cases}
				{\frac{4\pi-\delta}{4\pi-2\delta}}>1\ \text{as}\ \alpha_2\geq0,\\
				\\
				{\frac{p_0q_0}{q_0-1}}>1\ \ \text{as}\,-1<\alpha_2<0.
			\end{cases}
		\end{equation*}
		\par
		
		Define $\tau^n_2$ as the solution to the following equation,
		\begin{equation}\label{Sign-wf-4}
			\begin{cases}
				-\Delta_g\tau^n_2=-3\rho_2\ \text{in}\  B_{4r}(z),\\
				\tau^n_2|_{\partial B_{4r}(z)}=0.
			\end{cases}
		\end{equation}
		\par
		According to standard elliptic estimates, it yields that $\|\tau^n_2\|_{L^{\infty}(B_{4r}(z))}\leq C$ for some constant $C>0$. Define $\zeta_n:=v^n_1+2v^n_2-\overline{v}^n_1-2\overline{v}^n_2-\tau^n_1-\tau^n_2$, it follows from \eqref{Sign-wf-3} and \eqref{Sign-wf-4} that $\Delta_g\zeta_n=0$. By Harnack's inequality, we can acquire that 
		\begin{equation}\label{Sign-wf5}
			\begin{aligned}
				\|\zeta_n\|_{L^{\infty}(B_{2r}(z))}&\leq C\Big(\|v^n_1+2v^n_2-\overline{v}^n_1-2\overline{v}^n_2\|_{L^1(B_{4r}(z))}+\|\tau^n_1\|_{L^1(B_{4r}(z))}+\|\tau^n_2\|_{L^1(B_{4r}(z))}\Big)\\
				&\leq C\Big(1+\|\Delta_g(v^n_1+2v^n_2)\|_{L^1(\Sigma)}\Big)\\
				&\leq C.
			\end{aligned}
		\end{equation}
		\par
		
		Resorting to Jensen's inequality, we can have $ \overline{v}^n_i\leq\log\int_{\Sigma}e^{-H_i}e^{v^n_i}dV_g<+\infty$, which together with \eqref{Sign-wf5} implies $\|e^{v^n_i}\|_{L^p(B_{2r}(z))}$ is uniformly bounded, $i=1,2$. 
		Then we can attain that if $\alpha_2\geq0$, $\int_{B_{2r}(z)}e^{-pH_2}e^{pv^n_2}dV_g\leq C\int_{B_{2r}(z)}e^{pv^n_2}dV_g<+\infty$;\ If $-1<\alpha_2<0$, it follows from Holder inequality that
		\begin{equation*}
			\int_{B_{2r}(z)}e^{-p_0H_2}e^{p_0v^n_2}dV_g\leq\Big(\int_{B_{2r}(z)}e^{-p_0q_0H_2}dV_g\Big)^{\frac{1}{q_0}}\cdot\Big(\int_{B_{2r}(z)}e^{{\frac{p_0q_0}{q_0-1}}v^n_2}dV_g\Big)^{\frac{q_0-1}{q_0}}<+\infty.
		\end{equation*}
		\par
		
		Both situation indicate that $e^{-H_2}e^{v^n_2}$ is bounded in $L^s(B_{2r}(z))$ for some $s>1$. Successively, by applying \rm{$L^p$} estimate and Sobolev embedding theorem to \eqref{Sign-wf-3}, it yields that $\|\tau^n_1\|_{L^{\infty}(B_{r}(z))}$ is uniformly bounded. Recalling back the definition of $\zeta_n$, we can achieve that
		\begin{equation*}
			\|v^n_1+2v^n_2-\overline{v}^n_1-2\overline{v}^n_2\|_{L^{\infty}(B_r(z))}\leq C\ \ \text{for some constant}\ C>0.
		\end{equation*}
		\par
		
		Since $\overline{v}^n_i\leq C$ for some constant $C>0,\ i=1,2$, it suggests that $v^n_1+2v^n_2\leq C$ for some constant $C>0$, which contradicts with the assumption that $z$ is a common blow up point of $v^n_1$ and $v^n_2$.
		\par
		
		\noindent\textbf{Case 2.}
		Suppose only one component of $v^n_1$ and $v^n_2$ blows up. Without loss of generality, we can assume that $v^n_1$ blow up at the origin and $v^n_2$ is uniformly bounded from above. Similar to the idea in \cite[Lemma 2.5]{Zhu-a-w}.
		\par
		
		Let $\lambda^n_1:=\max\limits_{x\in\Sigma}v^n_1(x):=v^n_1(x^n_1)\rightarrow+\infty,\ x^n_1\rightarrow0,\ r^n_1:=e^{-\frac{\lambda^n_1}{2(1+\alpha_1)}}\rightarrow0$ and choose sufficient small $\tilde{r}>0$ such that $\overline{B}_{\tilde{r}}(0)\cap(S_1\cup S_2)=\{0\}$. Select an isothermal coordinate system centered at the origin, which satisfies $g=e^{\phi(x)}|dx|^2$ and $\phi(0)=0$. 
		\par
		\noindent\textbf{(i)} If ${\frac{|x^n_1|}{r^n_1}}\rightarrow+\infty$, then let $\mu_n:=|x^n_1|^{\alpha_1}e^{\frac{\lambda^n_1}{2}}$ and after scaling of $\varphi^n_1(x)=v^n_1(x^n_1+{\frac{x}{\mu_n}})-\lambda^n_1$, we can translate the second equation of \eqref{main-eq-1} as follows:
		\begin{equation*}
			-\Delta_{\mathbb{R}^2}\varphi^n_1(x)=2\rho^n_1h_1(x^n_1+{\frac{x}{\mu_n}})\Big|{\frac{x^n_1}{|x^n_1|}}+{\frac{x}{\mu_n|x^n_1|}}\Big|^{2\alpha_1}e^{{\varphi}^n_1-{\frac{\alpha_1}{2}}A(0)+\phi(x^n_1+{\frac{x}{\mu_n}})}+\mathcal{G}^n_1\ \text{in}\ B_{\tilde{r}}(0),
		\end{equation*}
		where $\mathcal{G}^n_1:=\Big(-\rho_2h_2(x^n_1+{\frac{x}{\mu_n}})\big|{\frac{x^n_1}{|x^n_1|}}+{\frac{x}{\mu_n|x^n_1|}}\big|^{2\alpha_2}{\frac{|x^n_1|^{2(\alpha_2-\alpha_1)}}{e^{\lambda^n_1}}}e^{v^n_2(x^n_1+{\frac{x}{\mu_n}})-{\frac{\alpha_2}{2}}A(0)}-\big(2\rho^n_1-\rho_2\big)(\frac{1}{\mu_n})^2\Big)e^{\phi(x^n_1+{\frac{x}{\mu_n}})}$.
		\par

		If $\alpha_2-\alpha_1\geq0$, $\frac{|x^n_1|^{2(\alpha_2-\alpha_1)}}{e^{\lambda^n_1}}\rightarrow0$;\ if $\alpha_2-\alpha_1\leq0$, then ${\frac{|x^n_1|^{2(\alpha_2-\alpha_1)}}{e^{\lambda^n_1}}}=\big(\frac{|x^n_1|}{r^n_1}\big)^{2(\alpha_2-\alpha_1)}\cdot(r^n_1)^{2(1+\alpha_2)}\rightarrow0$, which implies that $\mathcal{G}^n_1(x)\rightarrow0$ in $L^{\infty}_{\rm{loc}}(\mathbb{R}^2)$. Subsequently, by applying the elliptic estimates, we can obtain that $\varphi^n_1\rightarrow\varphi_1$ in $C^1_{\rm{loc}}(\mathbb{R}^2)$, where $\varphi_1$ satisfies the following equation,
		\begin{equation*}
			\begin{cases}
				-\Delta_{\mathbb{R}^2}\varphi_1=2\overline{\rho}_1h_1(0)e^{-\frac{\alpha_1}{2}A(0)}e^{\varphi_1}\ \ \text{in}\ \mathbb{R}^2,\\
				\varphi_1(0)=\lim\limits_{n\rightarrow+\infty}\varphi^n_1(0)=0,\\
				\int_{\mathbb{R}^2}e^{\varphi_1}dx<+\infty.
			\end{cases}
		\end{equation*}
		\par
		
		Suppose $h_1(0)\leq0$, then for any radius $R_0>0$, by strong maximum principle, we can obtain that $0=\lim\limits_{n\rightarrow+\infty}\varphi^n_1(0)=\varphi_1(0)\leq\max\limits_{\partial B_{R_0}(0)}\varphi_1(x)\leq0$. This result implies $\varphi_1$ is a constant in $\mathbb{R}^2$, which makes a contradiction to $\int_{\mathbb{R}^2}e^{\varphi_1}dx<+\infty$.
		\par
		
		\noindent\textbf{(ii)} If $0<{\frac{|x^n_1|}{r^n_1}}\leq C$ for some constant $C>0$, we can assume ${\frac{|x^n_1|}{r^n_1}}\rightarrow x_0$ and define $\tilde{v}^n_1:=v^n_1(x^n_1+r^n_1x)+2(1+\alpha_1)\log r^n_1$. Translate the first equation of \eqref{main-eq-1} as follows:
		\begin{equation*}
			-\Delta_{\mathbb{R}^2}\tilde{v}^n_1=2\rho^n_1h_1(x^n_1+r^n_1x)\Big|x+{\frac{x^n_1}{r^n_1}}\Big|^{2\alpha_1}e^{\tilde{v}^n_1-{\frac{\alpha_1}{2}}A(0)+\phi(x^n_1+r^n_1x)}+\mathcal{G}^n_2\ \ \text{in}\ B_{\tilde{r}}(0),
		\end{equation*}
		where $\mathcal{G}^n_2:=\Big(-\rho_2h_2(x^n_1+r^n_1x)\big|x+{\frac{x^n_1}{r^n_1}}\big|^{2\alpha_2}e^{v^n_2(x^n_1+r^n_1x)-{\frac{\alpha_2}{2}}A(0)}(r^n_1)^{2(1+\alpha_2)}-\big(2\rho^n_1-\rho_2\big)(r^n_1)^2\Big)e^{\phi(x^n_1+r^n_1x)}$.
		\par

		Since $\mathcal{G}^n_2(x)\rightarrow0$ in $L^q_{\rm{loc}}(\mathbb{R}^2)$ for $q\in(1,-{\frac{1}{\alpha_2}})$ as $\alpha_2<0$ and $q=+\infty$ as $\alpha_2\geq0$, then by elliptic estimates, $\tilde{v}^n_1\rightarrow\tilde{v}_1$ in $C^1_{\rm{loc}}(\mathbb{R}^2\backslash\{0\})\cap C^0_{\rm{loc}}(\mathbb{R}^2)\cap W^{2,s}_{\rm{loc}}(\mathbb{R}^2)$ for some $s>1$ as $\min\limits_{i=1,2}\{\alpha_i\}<0$ and $\tilde{v}^n_1\rightarrow\tilde{v}_1$\,in $C^1_{\rm{loc}}(\mathbb{R}^2)$ as $\min\limits_{i=1,2}\{\alpha_i\}\geq0$. Thus let $n\rightarrow+\infty$, we can derive the limit equation about $\tilde{v}_1$:
		\begin{equation*}
			\begin{cases}
				-\Delta_{\mathbb{R}^2}\tilde{v}_1=2\overline{\rho}_1h_1(0)e^{-\frac{\alpha_1}{2}A(0)}|x+x_0|^{2\alpha_1}e^{\tilde{v}_1}\ \text{in}\ \mathbb{R}^2,\\
				\tilde{v}_1(0)=\lim\limits_{n\rightarrow+\infty}\tilde{v}^n_1(0)=0,\\
				\int_{\mathbb{R}^2}|x+x_0|^{2\alpha_1}e^{\tilde{v}_1}dx<+\infty.
			\end{cases}
		\end{equation*}
		\par
		
		Suppose $h_1(0)\leq0$, then for any suitable $R_0$ larger than $|x_0|$, by revisiting strong maximum principle again, we can also lead to a similar contradiction.
	\end{proof}
	\par

	\begin{remark}
		Since $\rho^n_1h_1e^{-H_1}e^{v^n_1}$ is bounded in $L^1(\Sigma)$, we can assume that $\rho^n_1h_1e^{-H_1}e^{v^n_1}\rightharpoonup\mu_1$ in the sense of measure. In particular, due to the lemma \ref{Sign-wf}, it yields that for any blow-up point $x^0_1$ to $v^n_1,\ \mu_1(x^0_1)=\sigma_1(x^0_1)\geq0$. Similarly, $\mu_2(x^0_2)=\sigma_2(x^0_2)\geq0$ for any blow-up point $x^0_2$ of $v^n_2$.
		\par

		Before starting the following discussion, define the blow-up sets of $v^n_i,\ i=1,2$ as follows:
		\begin{equation*}
			\begin{aligned}
				S_1:&=\left\{ x\in\Sigma: \exists\,\{x^n_1\}\subset\Sigma, \ \text{such that}\ x^n_1\rightarrow x,\ v^n_1(x^n_1)\rightarrow+\infty
				\right\},\\
				S_2:&=\left\{ x\in\Sigma: \exists\,\{x^n_2\}\subset\Sigma, \ \text{such that}\ x^n_2\rightarrow x,\ v^n_2(x^n_2)\rightarrow+\infty
				\right\}.
			\end{aligned}
		\end{equation*}
	\end{remark}
	\par

	In the forthcoming lemma, we will present the Pohozaev identity associated with \eqref{main-eq-1}, which is recognized as a potent instrument in the analysis of blow-up phenomena. Before that, we also need to define $\alpha_i(z),\ i=1,2$ in the following way:
	\begin{equation}\label{DC-alph}
		\alpha_i(z)=
		\begin{cases}
		0,\ \ \text{as}\ z\neq0\\
		\alpha_i,\ \ \text{as}\ z=0. 	
		\end{cases}
	\end{equation}
	\par

	\begin{lemma}\label{PI} 
		If $(v^n_1,\,v^n_2)$ are the solution of \eqref{main-eq-1} and satisfy \eqref{EL}, then for any $z\in S_1\cup S_2$, we have that
		\begin{equation}\label{PI-3*}
			\sigma^2_1(z)+\sigma^2_2(z)-\sigma_1(z)\sigma_2(z)=4\pi\sigma_1(z)\big(1+\alpha_1(z)\big)+4\pi\sigma_2(z)\big(1+\alpha_2(z)\big).
		\end{equation}
	\end{lemma}
	\begin{proof}
		For a comprehensive proof, we refer readers to \cite{JLW} and \cite{LWZ} . For simplicity, we omit it here.
	\end{proof}
	\par

	\begin{lemma}\label{CPoint}
		Let $(v^n_1,v^n_2)$ be the blow-up solutions of \eqref{main-eq-1} and satisfy \eqref{EL},\ then $v^n_2$ is uniformly bounded from above and $v^n_1$ will blow up at only one point. In particular,
		\par
		
		\textbf{(i)} if $\alpha_1>0$, then $v^n_1$ can't blow up at the origin;
		\par
		
		\textbf{(ii)} if $\alpha_1<0$, then $v^n_1$ will blow up at the origin.
	\end{lemma}
	
	\begin{proof}
		Let $x^0_1$ be the blow-up point of $v^n_1$. Without loss of generality, take $\sigma_1(x^0_1)\geq\sigma_2(x^0_1)\geq0$ as an example. Then by Pohozaev identity \eqref{PI-3*}, it is clear that $\sigma_1(x^0_1)\geq4\pi(1+\alpha_1(x^0_1))$. Choose a sufficiently small $r>0$ such that $B_{r}(x^0_1)\cap\big(S_1\cup S_2\big)=\{x^0_1\}$.
		\par
		
		\noindent\textbf{Step 1.}
		First of all, we can claim that $\overline{v}^n_1\rightarrow-\infty$. Assume by contradiction that there exists a constant $C>0$ independent of $n$, such that $\overline{v}^n_1\geq-C$ for any $n\in\mathbb{N}$. By imposing Green representation formula on $v^n_1$ at the small disk $B_r(x^0_1)\backslash\{x^0_1\}$, it yields that there exists a constant $C>0$ independent of $x$ and $n$, such that 
		\begin{equation}\label{CPoint-0}
			\|v^n_1(x)-\overline{v}^n_1\|_{L^{\infty}(B_r(x^0_1)\backslash\{x^0_1\})}\leq C. 
		\end{equation}
		\par
		
		By exploiting \eqref{CPoint-0}, it results in $v^n_1(x)\geq-C$ for any $x\in B_{r}(x^0_1)\backslash\{x^0_1\}$. Then we can infer that $v^n_1$ is bounded on $\partial B_{r}(x^0_1)$. Let $\psi^n_1$ is a weak solution of the following equation,
		\begin{equation*}
			\begin{cases}
				-\Delta_g\psi^n_1=2\rho^n_1h_1(x)e^{-H_1}e^{v^n_1}-\rho_2h_2(x)e^{-H_2}e^{v^n_2}-(2\rho^n_1-\rho_2)\ \ \text{in}\ B_r(x^0_1),\\
				\\
				\psi^n_1|_{\partial B_{r}(x^0_1)}=-\|v^n_1\|_{L^{\infty}(\partial B_{r}(x^0_1))}>-\infty.
			\end{cases}
		\end{equation*}
		\par
		
		By maximum principle, it yields that $\psi^n_1(x)\leq v^n_1(x)$ which implies 
		\begin{equation}\label{CPoint-1}
			\int_{B_{r}(x^0_1)}e^{-H_1}e^{\psi^n_1}dV_g\leq\int_{B_{r}(x^0_1)}e^{-H_1}e^{v^n_1}dV_g\leq c_2.
		\end{equation}
		\par
		
		On the other hand, according to the proof in \cite[Lemma 5.6]{JW-A}, we can obtain that $\psi^n_1$ is uniformly bounded in $W^{1,p}(B_{r}(x^0_1))$ for any $1<p<2$. Here for the integrity of the whole narrative, we will prove it in detail. For any $1<p<2$, let $p':={\frac{p}{p-1}}$ and $B_r:=B_r(x^0_1)$. By the definition,
		\begin{equation*}
			\|\nabla_g\psi^n_1\|_{L^p}\leq\sup\left\{\Big|\int_{B_r}\nabla_g\psi^n_1\nabla_g\phi_0\,dV_g\Big|:\ \phi_0\in W^{1,p'}\,\text{satisfying}\,\int_{B_r}\phi_0\,dV_g=0,\, \|\phi_0\|_{W^{1,p'}(B_r)}=1\right\}.
		\end{equation*}
		\par
		
		Since $p'={\frac{p}{p-1}}>2$, then relying on Sobolev embedding theorem, it yields that $ \|\phi_0\|_{L^{\infty}(B_r(x^0_1))}\leq C$\ for some constant $C>0$. Hence,
		\begin{equation*}
			\begin{aligned}
				\Big|\int_{B_r(x^0_1)}\nabla_g\psi^n_1\cdot\nabla_g\phi_0 dV_g\Big|=&\Big|\int_{B_r(x^0_1)}\phi_0\cdot\Delta_g\psi^n_1dV_g\Big|\\
				\leq& \|\phi_0\|_{L^{\infty}(B_r)}\int_{B_r(x^0_1)}\Big|2\rho^n_1h_1e^{-H_1}e^{v^n_1}-\rho_2h_2e^{-H_2}e^{v^n_2}-(2\rho^n_1-\rho_2)\Big|dV_g\\
				\leq&C.
			\end{aligned}
		\end{equation*}
		\par
		
		This fact implies that $\|\nabla_g\psi^n_1\|_{L^p(B_r(x^0_1))}\leq C$ for some constant $C>0$, which also results in the boundedness of $\|\psi^n_1\|_{W^{1,p}(B_r(x^0_1))}$. Then there exists a function $\psi_1\in W^{1,p}(B_{r}(x^0_1))$ such that $\psi^n_1\rightharpoonup \psi_1$ in $W^{1,p}(B_{r}(x^0_1))$ and $\psi^n_1\rightarrow\psi_1$ in $L^q(B_{r}(x^0_1))$ for any $q\in(1,\,{\frac{2p}{2-p}}]$, where $\psi_1$ solves the following equation,
		\begin{equation*}
			\begin{cases}
				-\Delta_g\psi_1=\Big(2\mu_1(x^0_1)-\mu_2(x^0_1)\Big)\delta_{x^0_1}+O(1)\geq\mu_1(x^0_1)\delta_{x^0_1}+O(1),\\
				\psi_1|_{\partial B_{r}(x^0_1)}\ \ \text{is bounded}.
			\end{cases}
		\end{equation*}
		\par
		
		By exploiting Green representation formula and $\mu_1(x^0_1)\geq4\pi(1+\alpha_1(x^0_1))$, it yields that
		\begin{equation*}
			\psi_1\geq{\frac{\mu_1(x^0_1)}{2\pi}}\log{\frac{1}{\rm{dist}_g(x,\,x^0_1)}}+O(1)\geq2\big(1+\alpha_1(x^0_1)\big)\log{\frac{1}{\rm{dist}_g(x,\,x^0_1)}}+O(1)\ \ \text{in}\ B_r(x^0_1).
		\end{equation*}
		\par
		
		Following from Fatou lemma, it indicates that
		\begin{equation*}
			\begin{aligned}
				\lim\inf\limits_{n\rightarrow+\infty}\int_{B_{r}(x^0_1)}e^{-H_1}e^{\psi^n_1}dV_g\geq\int_{B_{r}(x^0_1)}e^{-H_1}e^{\psi_1}dV_g\geq C\int_{B_{r}(x^0_1)}{\frac{1}{\rm{dist}_g(x,\,x^0_1)^2}}dV_g\rightarrow+\infty,
			\end{aligned}
		\end{equation*}
		\par
		
		which contradicts to \eqref{CPoint-1}, then our claim has been confirmed. In the meanwhile, by reintroducing \eqref{CPoint-0}, we can also deduce that $v^n_1(x)\rightarrow-\infty$ in $B_r(x^0_1)\backslash\{x^0_1\}$.
		\par
		
		\noindent\textbf{Step 2.} By resorting to the Green representation formula again, we can attain the boundary oscillation condition for $v^n_1$ on any compact subset of $\Sigma\backslash(S_1\cup S_2)$. That is, for any fixed $\tilde{x}_0\in\Sigma\backslash(S_1\cup S_2)$, define $0<r_0<{\frac{1}{2}}\rm{\dist_g}(\tilde{x}_0,\ S_1\cup S_2)$ and there exists a constant $C>0$ such that
		\begin{equation*}
			\big|v^n_1(x)-v^n_1(y)\big|\leq C,\  \forall\  x,\,y\in\partial B_{r_0}(\tilde{x}_0).
		\end{equation*}
		\par
		
		Owing to the arbitrary of $\tilde{x}_0$ in $\Sigma\backslash(S_1\cup S_2)$ and $v^n_1(x)\rightarrow-\infty$ for any $x\in B_r(x^0_1)\backslash\{x^0_1\}$ in Step 1, it yields that $v^n_1\rightarrow-\infty$ in any compact subset of $\Sigma\backslash(S_1\cup S_2)$, which implies that
		\begin{equation}\label{CPoint-10}
			4\pi(1+\overline{\alpha}_1)\geq\lim\limits_{n\rightarrow+\infty}\int_{B_r(x^0_1)}h_1e^{-H_1}e^{v^n_1}dV_g=\sigma_1(x^0_1).
		\end{equation}
		\par
		
		Combining \eqref{CPoint-10} and the assumption that $\sigma_1(x^0_1)\geq4\pi\big(1+\alpha_1(x^0_1)\big)$, we can deduce that
		\begin{equation*}
			\sigma_1(x^0_1)=4\pi(1+\alpha_1(x^0_1))=4\pi(1+\overline{\alpha}_1),
		\end{equation*}
which indicates that $v^n_1$ will blow up at the only one point.
\par

 After substituting $\sigma_1(x^0_1)=4\pi(1+\alpha_1(x^0_1))$ into the Pohozaev identity\,\eqref{PI-3*}, it yields $\sigma_2(x^0_1)=0$. In addition, we can also conclude that
		\par
		
		\textbf{(i)} If $\alpha_1>0$, then $x^0_1\neq0$, which indicates that $v^n_1$ can't blow up at the origin.
		\par 
		 
		\textbf{(ii)} If $\alpha_1<0$, then $x^0_1=0$, which implies that $v^n_1$ will blow up at the origin.
		\par

		On the other hand, let $x^0_2$ be the blow-up point of $v^n_2$. On the hypothesis that $\sigma_2(x^0_2)\geq\sigma_1(x^0_2)\geq0$, similarly we can obtain that
		\begin{equation*}
		\sigma_2(x^0_2)=4\pi(1+\alpha_2(x^0_2))=4\pi(1+\overline{\alpha}_2),	
		\end{equation*}
which leads to a contradiction due to $0<\rho_2<4\pi(1+\overline{\alpha}_2)$. This fact indicates that the situation of $v^n_2$ blowing up can't occur.
\par

Thus combining the two blow-up scenario together, the proof has been completed. 
\end{proof}
\par

\begin{remark}\label{QuantiS-R}
	Let $w_n$ be the solution of following mean field equation,
	\begin{equation*}
		\begin{cases}
			-\Delta_gw_n=\rho_2\big(h_2e^{-H_2}e^{v^n_2}-1\big)\\
			\int_{\Sigma}w_ndV_g=0.
		\end{cases}
	\end{equation*}
	\par
	
	Define $\mathcal{F}(x):=\rho_2\big(h_2e^{-H_2}e^{v^n_2}-1\big)$, since $v^n_2$ is uniformly bounded from above, we can obtain that $\mathcal{F}(x)\in L^q(\Sigma)$ for $q\in(1,\,-{\frac{1}{\alpha_2}})$ as $-1<\alpha_2<0$ and $q=+\infty$ as $\alpha_2\geq0$. Then by standard elliptic theory, it indicates that $w_n$ is uniformly bounded in $W^{2,q}_{\rm{loc}}(\Sigma)$.
	\par

	Furthermore, there exists some function $w_0\in W^{2,q}_{\rm{loc}}(\Sigma)$ such that $w_n\rightarrow w_0$ in $C^1_{\rm{loc}}(\Sigma\backslash\{0\})\cap C^0_{\rm{loc}}(\Sigma)\cap W^{2,q}_{\rm{loc}}(\Sigma)$ for $q\in(1,-{\frac{1}{\alpha_2}})$ as $-1<\alpha_2<0$ and $w_n\rightarrow w_0$ in $C^1(\Sigma)$ as $\alpha_2\geq0$.
	\par

    For the convenience of the next lemma, we will define $\overline{\rho}_1(x^0_1)$ in the following form:
    \begin{equation*}
	\overline{\rho}_1(x^0_1):=
	\begin{cases}
		\overline{\rho}_1=4\pi(1+\alpha_1),\ \ \text{as}\ x^0_1=0,\\
		\\
		\overline{\rho}_1=4\pi,\ \ \text{as}\ x^0_1\neq0.
	\end{cases}
    \end{equation*}
\end{remark}

\begin{prop}\label{WeakConver}
	For any $1<p<2$, with $w_n$ defined in Remark \ref{QuantiS-R}, we have that $v^n_1-\overline{v}^n_1+w_n\rightharpoonup2\overline{\rho}_1(x^0_1)G_{x^0_1}$ and $v^n_2-\overline{v}^n_2-2w_n\rightharpoonup-\overline{\rho}_1(x^0_1)G_{x^0_1}$ weakly in $W^{1,p}(\Sigma)$.
	\par
	
	In addition, $v^n_1-\overline{v}^n_1+w_n\rightarrow2\overline{\rho}_1(x^0_1)G_{x^0_1}$ and $v^n_2-\overline{v}^n_2-2w_n\rightarrow-\overline{\rho}_1(x^0_1)G_{x^0_1}$ strongly in $C^1_{\rm{loc}}(\Sigma\backslash\{0,x^0_1\})\cap C^0_{\rm{loc}}(\Sigma\backslash\{x^0_1\})\cap W^{2,s}_{\rm{loc}}(\Sigma\backslash\{x^0_1\})$ for any $s\in(1,-{\frac{1}{\alpha_1}})$ as $-1<\alpha_1<0$ and in $C^1_{\rm{loc}}(\Sigma\backslash\{x^0_1\})$ as $\alpha_1\geq0$.
	\par	
\end{prop}

\begin{proof}
	It follows from Lemma \ref{CPoint} that $h_1e^{-H_1}e^{v^n_1}\rightarrow\delta_{x^0_1}$ in the sense of measure. Subsequently, anagolous to the argument in \cite[Lemma 5.6]{JW-A}, for any $1<p<2$, let $p'={\frac{p}{p-1}}$, then for any $\phi_0\in W^{1,p'}(\Sigma)$ satisfying $\int_{\Sigma}\phi_0\,dV_g=0\ \text{and}\,\|\phi_0\|_{W^{1,p'}(\Sigma)}=1$, we can deduce that 
	\begin{equation*}
		\begin{aligned}
			\Big|\int_{\Sigma}\nabla_g\big(v^n_1-\overline{v}^n_1+w_n-2\overline{\rho}_1(x^0_1)G_{x^0_1}\big)\nabla_g\phi_0\,dV_g\Big|&\leq C\int_{\Sigma}\Big|2\rho^n_1h_1e^{-H_1}e^{v^n_1}-2\overline{\rho}_1(x^0_1)\delta_{x^0_1}\Big|dV_g+o_n(1)\rightarrow0,
		\end{aligned}
	\end{equation*}
which implies that $\|\nabla_g\big(v^n_1-\overline{v}^n_1+w_n-2\overline{\rho}_1(x^0_1)G_{x^0_1}\big)\|_{L^p(\Sigma)}\rightarrow0$. This fact together with Poincare inequality yields that $v^n_1-\overline{v}^n_1+w_n\rightharpoonup2\overline{\rho}_1(x^0_1)G_{x^0_1}$ weakly in $W^{1,p}(\Sigma)$. 
	\par

	In addition, for any $\Omega\subset\subset\Sigma\backslash\{x^0_1\}$, by applying standard elliptic estimates, it is obvious that $v^n_1-\overline{v}^n_1+w_n\rightarrow2\overline{\rho}_1(x^0_1)G_{x^0_1}$ in $C^1(\Omega\backslash\{0\})\cap C^0(\Omega)\cap W^{2,s}(\Omega)$ for any $s\in(1,-{\frac{1}{\alpha_1}})$ as $-1<\alpha_1<0$ and in $C^1_{\rm{loc}}(\Omega)$ as $\alpha_1\geq0$. The same applies to $v^n_2-\overline{v}^n_2-2w_n$. Thus, this lemma has been confirmed.
\end{proof}
\par

			\begin{remark}\label{AF}
			According to Remark \ref{QuantiS-R} and Proposition \ref{WeakConver}, we can attain the limit Liouville-type equation which $w_0$ satisfies. That is, $w_0$ is the solution of the following Liouville equation,
			\begin{equation}\label{AF-1}
					\begin{cases}
						-\Delta_gw_0=\rho_2(\frac{h_2e^{-4\pi\alpha_2G_0-\rho_1G_{x^0_1}}}{\int_{\Sigma}h_2e^{-4\pi\alpha_2G_0-\rho_1G_{x^0_1}+2w_0}dV_g}e^{2w_0}-1)\\
						\int_{\Sigma}w_0dV_g=0.
					\end{cases}
			\end{equation}
			\par

			Since $0<\rho_2<4\pi(1+\overline{\alpha}_2)$, then by resorting to \cite[Theorem 1.9]{BGJM}, $w_0$ is the unique solution of \eqref{AF-1}.
		\end{remark}
		\par
	
	    In the sequel, after scaling, we will explore the limit formula of $v^n_1$ near the blow-up point. 
		\begin{lemma}\label{Blow-upP}   
			Let $\lambda^n_1:=\max\limits_{x\in\Sigma}v^n_1=v^n_1(x^n_1)\rightarrow+\infty,\ x^n_1\rightarrow x^0_1$ and define $\tilde{v}^n_1:=v^n_1(x^n_1+r^n_1x)+2\big(1+\alpha_1(x^0_1)\big)\log r^n_1,\,r^n_1:=e^{-\frac{\lambda^n_1}{2(1+\alpha_1(x^0_1))}}$.
			\par

		If $x^0_1=0$, then $\tilde{v}^n_1\rightarrow\tilde{v}_1$ in $C^1_{\rm{loc}}(\mathbb{R}^2\backslash\{0\})\cap C^0_{\rm{loc}}(\mathbb{R}^2)\cap W^{2,s}_{\rm{loc}}(\mathbb{R}^2)$ for some $s>1$ as $\min\limits_{i=1,2}\{\alpha_i\}<0$ and $\tilde{v}^n_1\rightarrow\tilde{v}_1$ in $C^1_{\rm{loc}}(\mathbb{R}^2)$ as $\min\limits_{i=1,2}\{\alpha_i\}=0$, where
			\begin{equation*}
				\tilde{v}_1(x)=-2\log\big(1+{\frac{{\pi}h_1(0)e^{-\frac{\overline{\alpha}_1}{2}A(0)} }{1+\overline{\alpha}_1}}|x|^{2(1+\overline{\alpha}_1)}\big);
			\end{equation*}
			\par

		If $x^0_1\neq0$, then $\tilde{v}^n_1\rightarrow\tilde{v}_1$ in $C^1_{\rm{loc}}(\mathbb{R}^2)$, where
			\begin{equation*}
				\tilde{v}_1(x)=-2\log\big(1+\pi h_1(x^0_1)e^{-4\pi\alpha_1G_0(x^0_1)}|x|^2\big).
			\end{equation*}

		\end{lemma}

		\begin{proof}
			Select an isothermal coordinate system which is centered at $x^0_1$ and satisfies $g=e^{\phi(x)}|dx|^2$ and $\phi(x^0_1)=0$. \textbf{Case (i)} If $v^n_1$ blows up at the origin, i.e., $x^0_1=0$.
			\par

			Recalling back lemma \ref{CPoint}, it is apparent that $\alpha_1\leq0$ and $e^{v^n_2}$ is uniformly bounded from above. Then by applying Theorem 4.1 in\,\cite{BT}, it yields that $\lim\limits_{n\rightarrow+\infty}{\frac{|x^n_1|}{r^n_1}}=0$. By the definition of $\tilde{v}^n_1(x)$, we can translate the first equation of \eqref{main-eq-1} as follows:
			\begin{equation*}
				-\Delta_{\mathbb{R}^2}\tilde{v}^n_1=2\rho^n_1h_1(x^n_1+r^n_1x)\Big|x+{\frac{x^n_1}{r^n_1}}\Big|^{2\overline{\alpha}_1}e^{\tilde{v}^n_1-{\frac{\overline{\alpha}_1}{2}}A(0)+\phi(x^n_1+r^n_1x)}+\mathcal{G}_n(x)\ \ \text{in}\ B_{\tilde{r}}(0),
			\end{equation*}
			where
			$\mathcal{G}_n:=\Big(-\rho_2h_2(x^n_1+r^n_1x)\Big|x+{\frac{x^n_1}{r^n_1}}\Big|^{2\alpha_2}e^{v^n_2(x^n_1+r^n_1x)-{\frac{\alpha_2}{2}}A(0)}(r^n_1)^{2(1+\alpha_2)}-\big(2\rho^n_1-\rho_2\big)\cdot(r^n_1)^2\Big)e^{\phi(x^n_1+r^n_1x)}$.
			\par

			Since $\mathcal{G}_n\rightarrow0\,\ \text{in}\, L^q_{\rm{loc}}(\mathbb{R}^2)$ for $q\in\,(1,-{\frac{1}{\alpha_2}})$ as $\alpha_2<0$ and $q=+\infty$ as $\alpha_2\geq0$, then by elliptic estimates, there exists some function $\tilde{v}_1\in W^{2,s}_{\rm{loc}}(\mathbb{R}^2)$ such that $\tilde{v}^n_1\rightarrow\tilde{v}_1$ in $C^1_{\rm{loc}}(\mathbb{R}^2\backslash\{0\})\cap C^0_{\rm{loc}}(\mathbb{R}^2)\cap W^{2,s}_{\rm{loc}}(\mathbb{R}^2)$ for some $s>1$ as $\min\limits_{i=1,2}\{\alpha_i\}<0$ and $\tilde{v}^n_1\rightarrow\tilde{v}_1$ in $C^1_{\rm{loc}}(\mathbb{R}^2)$ as $\min\limits_{i=1,2}\{\alpha_i\}=0$.
			\par

			Let $n\rightarrow+\infty,$ we can arrive at the limit equation about $\tilde{v}_1$:
			\begin{equation*}
				\begin{cases}
					-\Delta_{\mathbb{R}^2}\tilde{v}_1=8\pi(1+\alpha_1)h_1(0)e^{-\frac{\overline{\alpha}_1}{2}A(0)}|x|^{2\overline{\alpha}_1}e^{\tilde{v}_1}\ \ \text{in}\ \mathbb{R}^2,\\
					\\
					\int_{\mathbb{R}^2}|x|^{2\overline{\alpha}_1}e^{\tilde{v}_1}dx<+\infty.
				\end{cases}
			\end{equation*}
			\par

			Relying on the classification result in\,\cite{Chen-Li}\,or\,\cite{PT}, we can attain
			\begin{equation}\label{Blow-upP-1}
				\tilde{v}_1(x)=-2\log\big(1+{\frac{\pi h_1(0)e^{-\frac{\overline{\alpha}_1}{2}A(0)}}{1+\overline{\alpha}_1}}|x|^{2(1+\overline{\alpha}_1)}\big),\ \  h_1(0)e^{{-\frac{\overline{\alpha}_1}{2}}A(0)}\int_{\mathbb{R}^2}|x|^{2\overline{\alpha}_1}e^{\tilde{v}_1}dx=1.
			\end{equation}
			\par

			\textbf{Case(ii)} If $v^n_1$ blows up at $x^0_1\neq0$, it follows from Lemma \ref{CPoint} that $\alpha_1\geq0$. Transform the first equation of \eqref{main-eq-1} as follows: 
			\begin{equation*}
				-\Delta_{\mathbb{R}^2}\tilde{v}^n_1=2\rho^n_1h_1(x^n_1+r^n_1x)e^{-4\pi\alpha_1G_0(x^n_1+r^n_1x)}e^{\tilde{v}^n_1(x)+\phi(x^n_1+r^n_1x)}+\mathcal{G}_n(x)\ \ \text{in}\ B_{\tilde{r}}(0),
			\end{equation*}
			where $\mathcal{G}_n:=\Big(-\rho_2h_2(x^n_1+r^n_1x)e^{-4\pi\alpha_2G_0(x^n_1+r^n_1x)}e^{v^n_2(x^n_1+r^n_1x)}(r^n_1)^2-\big(2\rho^n_1-\rho_2\big)(r^n_1)^2\Big)e^{\phi(x^n_1+r^n_1x)}$.
			\par

			Since $\mathcal{G}_n\rightarrow0\,\ \text{in}\, L^{\infty}_{\rm{loc}}(\mathbb{R}^2)$, then by elliptic estimates, there exists some function $\tilde{v}_1\in W^{2,s}_{\rm{loc}}(\mathbb{R}^2)$ for any $s>1$ such that $\tilde{v}^n_1\rightarrow\tilde{v}_1$ in $C^1_{\rm{loc}}(\mathbb{R}^2)$.
			\par

			Let $n\rightarrow+\infty,$ we can derive the limit equation about $\tilde{v}_1$:
			\begin{equation*}
				\begin{cases}
					-\Delta_{\mathbb{R}^2}\tilde{v}_1={8\pi}h_1(x^0_1)e^{-4\pi\alpha_1G_0(x^0_1)}e^{\tilde{v}_1}\ \ \text{in}\ \mathbb{R}^2,\\
					\\
					\int_{\mathbb{R}^2}e^{\tilde{v}_1}dx<+\infty.
				\end{cases}
			\end{equation*}			
			\par

			By applying the classification result of Chen and Li in\,\cite{Chen-Li}, it yields that
			\begin{equation*}
				\tilde{v}_1(x)=-2\log\big(1+\pi h_1(x^0_1)e^{-4\pi\alpha_1G_0(x^0_1)}|x|^2\big),\ \ h_1(x^0_1)e^{-4\pi\alpha_1G_0(x^0_1)}\int_{\mathbb{R}^2}e^{\tilde{v}_1}dx=1.			
			\end{equation*}	
			\par
			
			Thus this lemma has been confirmed.
		\end{proof}

\begin{remark}
It follows from Lemma\ref{Blow-upP}, we can generalize $\tilde{v}_1(x)$ in the following form:
\begin{equation*}
		 	\tilde{v}_1(x):=-2\log\big(1+{\frac{\pi\tilde{h}_1(x^0_1)}{1+\alpha_1(x^0_1)}}|x|^{2(1+\alpha_1(x^0_1))}\big),
\end{equation*}
where $\tilde{h}_1(x^0_1):=h_1(0)e^{-\frac{\overline{\alpha}_1}{2}A(0)}$ as $x^0_1=0$ and $\tilde{h}_1(x^0_1):=h_1(x^0_1)e^{-4\pi G_0(x^0_1)}$ as $x^0_1\neq0$.
\end{remark}	
\par

Subsequently, let $\zeta^n_1:=v^n_1+w_n$ and $\zeta^n_2:=v^n_2-2w_n$ and we will reveal the equivalent relationship between $\zeta^n_1-\overline{\zeta}^n_1$ and $\zeta^n_2-\overline{\zeta}^n_2$. 
\par
\begin{lemma}\label{QuantiS}
For any blow-up solution $(v^n_1,\,v^n_2)$ of system \eqref{main-eq-1}, we have that
\begin{equation*}
\zeta^n_1-\overline{\zeta}^n_1+2\zeta^n_2-2\overline{\zeta}^n_2=0,\,\ \text{for any}\ x\in\Sigma.
\end{equation*}
\end{lemma}
\begin{proof}
			Since $\zeta^n_1$ and $\zeta^n_2$ satisfy the following equations separately,
			\begin{equation*}
				-\Delta_g\zeta^n_1={2\rho^n_1}h_1e^{-H_1}e^{v^n_1}-2\rho^n_1,\ 
				-\Delta_g\zeta^n_2=-\rho^n_1h_1e^{-H_1}e^{v^n_1}+\rho^n_1.
			\end{equation*}
			\par
			
			By combining the two equations, we can acquire $-\Delta_g(\zeta^n_1+2\zeta^n_2-\overline{\zeta}^n_1-2\overline{\zeta}^n_2)=0$, which implies that
			\begin{equation}\label{QuantiS-3}
				\int_{\Sigma}\big|\Delta_g(\zeta^n_1+2\zeta^n_2-\overline{\zeta}^n_1-2\overline{\zeta}^n_2)\big|^2dV_g=0.
			\end{equation}
			\par
			
			On one hand, multiplying $\varphi:=\zeta^n_1+2\zeta^n_2-\overline{\zeta}^n_1-2\overline{\zeta}^n_2$ in both sides of\,$-\Delta_g(\zeta^n_1+2\zeta^n_2-\overline{\zeta}^n_1-2\overline{\zeta}^n_2)=0$ and integrating over $\Sigma$, we can have that
			\begin{equation}\label{QuantiS-4}
				\|\nabla_g(\zeta^n_1+2\zeta^n_2-\overline{\zeta}^n_1-2\overline{\zeta}^n_2)\|_{L^2(\Sigma)}=0.
			\end{equation}
			\par
			
			On the other hand, by utilizing Poincar\'e inequality, we can arrive at
			\begin{equation}\label{QuantiS-5}
				\|\zeta^n_1+2\zeta^n_2-\overline{\zeta}^n_1-2\overline{\zeta}^n_2\|_{L^2(\Sigma)}\leq C\|\nabla_g\big(\zeta^n_1+2\zeta^n_2-\overline{\zeta}^n_1-2\overline{\zeta}^n_2\big)\|_{L^2(\Sigma)}=0.
			\end{equation}
			\par
			
			Combining \eqref{QuantiS-3},\eqref{QuantiS-4} and \eqref{QuantiS-5} together, we can deduce that $\|\zeta^n_1+2\zeta^n_2-\overline{\zeta}^n_1-2\overline{\zeta}^n_2\|_{W^{2,2}(\Sigma)}=0.$ After applying Sobolev embedding theorem, it yields that \
			\begin{equation*}
				\zeta^n_1+2\zeta^n_2-\overline{\zeta}^n_1-2\overline{\zeta}^n_2=0\ \ \text{in}\ C^0(\Sigma).
			\end{equation*}
			\par
			
			Thus, this lemma has been proved.
		\end{proof}
		\par

\begin{remark}
Actually by recalling back Lemma\,\ref{QuantiS}, we can obtain that $v^n_2(x)=-{\frac{1}{2}}v^n_1+{\frac{1}{2}}\overline{v}^n_1+{\frac{3}{2}}w_n(x)+\overline{v}^n_2$.
\end{remark}
		\par

	\section{Lower boundedness}
	In this section, under the situation that $v^n_1$ blows up at most at the point $x^0_1$ and $v^n_2$ is uniformly bounded from above, we will derive the lower boundedness of the infimum value of the energy function $J_{\overline{\rho}_1}$.
	\begin{lemma}\label{LB}
		Let $(v^n_1,\,v^n_2)$ be the blow-up solution of \eqref{main-eq-1}.
		\par
		
	 \textbf{(i)} If $\alpha_1>0$, we can obtain that
		\begin{equation*}
			\begin{aligned}
				\inf\limits_{H^1(\Sigma)}J_{\overline{\rho}_1}\geq&\min\limits_{x\in\Sigma_+\backslash\{0\}}\Big(\int_{\Sigma}|\nabla_gw_0|^2dV_g-4\pi\int_{\Sigma}\nabla_gG_{x}\nabla_gw_0dV_g-\rho_2\log\int_{\Sigma}h_2e^{-4\pi\alpha_2G_0-{4\pi}G_{x}}e^{2w_0}dV_g\\
				&\quad-2{\pi}A(x)+{4\pi}w_0(x)-4\pi\log\big( h_1(x)e^{-4\pi\alpha_1G_0(x)}\big)\Big)-4\pi\log\pi-4\pi.
			\end{aligned}
		\end{equation*}
		\par

		\textbf{(ii)}\ If $\alpha_1<0$, then
	\begin{equation*}
		\begin{aligned}
			\inf\limits_{H^1(\Sigma)}J_{\overline{\rho}_1}\geq&\int_{\Sigma}|\nabla_gw_0|^2dV_g-\overline{\rho}_1\int_{\Sigma}\nabla_gG_0\nabla_gw_0dV_g-\rho_2\log\int_{\Sigma}h_2e^{-4\pi\alpha_2G_0-\overline{\rho}_1G_0}e^{2w_0}dV_g\\
			&\quad-{\frac{(\overline{\rho}_1)^2}{8\pi}}A(0)+\overline{\rho}_1w_0(0)-\overline{\rho}_1\log\big({\frac{{h_1(0)e^{-\frac{\overline{\alpha}_1}{2}A(0)}}}{1+\overline{\alpha}_1}}\big)-\overline{\rho}_1\log\pi-\overline{\rho}_1.
		\end{aligned}
	\end{equation*}
	\par

	\textbf{(iii)}\ If $\alpha_1=0$, then
	\begin{equation*}
		\begin{aligned}
			\inf\limits_{H^1(\Sigma)}J_{\overline{\rho}_1}\geq&\min\limits_{x\in\Sigma_+}\Big(\int_{\Sigma}|\nabla_gw_0|^2dV_g-4\pi\int_{\Sigma}\nabla_gG_{x}\nabla_gw_0dV_g-\rho_2\log\int_{\Sigma}h_2e^{-4\pi\alpha_2G_0-{4\pi}G_{x}}e^{2w_0}dV_g\\
			&\quad-2{\pi}A(x)+{4\pi}w_0(x)-4\pi\log{h_1(x)}\Big)-4\pi\log\pi-4\pi,
		\end{aligned}
	\end{equation*}
where $\Sigma_+:=\{x\in\Sigma:\ h_1(x)>0\}$.	
	\end{lemma}
	
	\begin{proof}
		In order to compute gradient term $\int_{\Sigma}Q(v^n_1,v^n_2)dV_g$, we can decompose the Riemann surface $\Sigma$ into three parts and calculate separately.
		\par
		For the part outside the neighborhood of blow-up point $\Sigma\backslash B_\delta(x^n_1)$, by the fact $v^n_1-\overline{v}^n_1+w_n\rightarrow2\overline{\rho}_1(x^0_1)G_{x^0_1}$ and $v^n_2-\overline{v}^n_2-2w_n\rightarrow-\overline{\rho}_1(x^0_1)G_{x^0_1}$ in $C^0_{\rm{loc}}(\Sigma\backslash\{x^0_1\})\cap W^{2,s}_{\rm{loc}}(\Sigma\backslash\{x^0_1\})$ for some $s>1$ in Proposition \ref{WeakConver}, we can acquire that
		\begin{equation}\label{LB-O}
			\begin{aligned}
				&\int_{\Sigma\backslash B_{\delta}(x^n_1)}Q(v^n_1,\,v^n_2)dV_g\\
				=&\big(\overline{\rho}_1(x^0_1)\big)^2\int_{\Sigma\backslash B_{\delta}(x^n_1)}|\nabla_gG_{x^0_1}|^2dV_g+\int_{\Sigma\backslash B_{\delta}(x^n_1)}|\nabla_gw_0|^2dV_g-\overline{\rho}_1(x^0_1)\int_{\Sigma\backslash B_{\delta}(x^n_1)}\nabla_gG_{x^0_1}\nabla_gw_0dV_g\\
				=&\int_{\Sigma}|\nabla_gw_0|^2dV_g-\overline{\rho}_1(x^0_1)\int_{\Sigma}\nabla_gG_{x^0_1}\nabla_gw_0dV_g-{\frac{\big(\overline{\rho}_1(x^0_1)\big)^2}{2\pi}}\log\delta+{\frac{\big(\overline{\rho}_1(x^0_1)\big)^2}{8\pi}}A(x^0_1)+o_{\delta}(1)+o_n(1),
			\end{aligned}
		\end{equation}
where $\overline{\rho}_1(x^0_1)=4\pi(1+\overline{\alpha}_1)$ as $x^0_1=0$ and $\overline{\rho}_1(x^0_1)=4\pi$ as $x^0_1\neq0$.
\par	

		For the part near the blow-up point $B_{Rr^n_1}(x^n_1)$, by exploiting $v^n_2(x)=-{\frac{1}{2}}v^n_1+{\frac{1}{2}}\overline{v}^n_1+{\frac{3}{2}}w_0(x)+\overline{v}^n_2$ in computation of $\int_{B_{Rr^n_1}(x^n_1)}Q(v^n_1,v^n_2)dV_g$, it yields that, 
		\begin{equation}\label{LB-B}
			\begin{aligned}
				\int_{B_{Rr^n_1}(x^n_1)}Q(v^n_1,v^n_2)dV_g=&{\frac{1}{4}}\int_{B_{Rr^n_1}(x^n_1)}|\nabla_gv^n_1|^2dV_g+{\frac{3}{4}}\int_{B_{Rr^n_1}(x^n_1)}|\nabla_gw_0|^2dV_g\\
				=&{\frac{1}{4}}\int_{B_{R}(0)}|\nabla_g\tilde{v}^n_1|^2e^{\psi(x^n_1+r^n_1x)}dx+o_R(1)+o_n(1).\\
				=&{\frac{\pi}{2}}\int^R_0\Big|{\frac{-4(1+\alpha_1(x^0_1))\pi{\tilde{h}_1(x^0_1) }r^{(2\alpha_1(x^0_1)+1)}}{1+\alpha_1(x^0_1)+\pi\tilde{h}_1(x^0_1) r^{2(1+\alpha_1(x^0_1))}}}\Big|^2rdr+o_R(1)+o_n(1)\\
				=&\overline{\rho}_1(x^0_1)\log\big(1+{{\frac{\pi \tilde{h}_1(x^0_1)}{1+\alpha_1(x^0_1)}}R^{2(1+\alpha_1(x^0_1))}}\big)-\overline{\rho}_1(x^0_1)+o_R(1)+o_n(1),
			\end{aligned}
		\end{equation}
		where $\tilde{h}_1(x^0_1)=h_1(0)e^{-{\frac{\alpha_1}{2}}A(0)},\,\alpha_1(x^0_1)=\overline{\alpha}_1$ for $x^0_1=0$ and $\tilde{h}_1(x^0_1):=h_1(x^0_1)e^{-4\pi\alpha_1G_0(x^0_1)},\,\alpha_1(x^0_1)=0$ for $x^0_1\neq0$.

		For the neck part $N_{\delta,n}:=B_{\delta}(x^n_1)\backslash B_{Rr^n_1}(x^n_1)$, following from the idea in \cite{LiZhu}, we can define
		\begin{equation}\label{LB-N-0}
			\theta^n_1:=
			a^n_1-b^n_1+2\big(1+\alpha_1(x^0_1)\big)\log r^n_1+\overline{v}^n_1-w_0(x),
		\end{equation}
		where $a^n_1:=\inf\limits_{\partial B_{Rr^n_1}(x^n_1)}\big(v^n_1+w_0\big),\ b^n_1:=\sup\limits_{\partial B_{\delta}(x^n_1)}\big(v^n_1+w_0\big)$. Recalling back $v^n_1-\overline{v}^n_1+w_0\rightarrow2\overline{\rho}_1(x^0_1)G_{x^0_1}\ \text{in}\ C^0_{\rm{loc}}(\Sigma\backslash\{x^0_1\})$ and $v^n_1(x^n_1+r^n_1x)+2(1+\alpha_1(x^0_1))\log r^n_1\rightarrow\tilde{v}_1$ in $C^0_{\rm{loc}}(\mathbb{R}^2)$, we can deduce that
		\begin{equation}\label{LB-N-1*}
			\begin{aligned} \theta^n_1\rightarrow\inf\limits_{|x|=R}\tilde{v}_1(x)-2\overline{\rho}_1(x^0_1)\sup\limits_{\partial B_{\delta}(x^0_1)}G_{x^0_1}.
			\end{aligned}
		\end{equation}	
		\par

		Let $Z^n_1:=\max\left\{\min\{v^n_1+w_0,a^n_1\},b^n_1\right\}$ and it is apparent that $\ Z^n_1|_{\partial B_{Rr^n_1}(x^n_1)}=a^n_1,\ Z^n_1|_{\partial B_{\delta}(x^n_1)}=b^n_1$. Selecting an isothermal coordinate system centered at $x^n_1$ and employing the Dirichlet principle in\cite{LiZhu}, it indicates
		\begin{equation}\label{LB-N-1}
			\begin{aligned}
				\int_{N_{\delta,n}}|\nabla_g(v^n_1+w_0)|^2dV_g&\geq\int_{N_{\delta,n}}|\nabla_gZ^n_1|^2dV_g\\
				&=\int_{B_{\delta}(0)\backslash B_{Rr^n_1}(0)}|\nabla_{\mathbb{R}^2}Z^n_1|^2dx\\
				&\geq\inf\limits_{\chi|{\partial B_{Rr^n_1}(0)}=a^n_1,\ \chi|{\partial B_{\delta}(0)}=b^n_1} 
				\int_{B_{\delta}(0)\backslash B_{Rr^n_1}(0)}|\nabla_{\mathbb{R}^2}\chi|^2dx,
			\end{aligned}
		\end{equation}
		\par
		
		where the infimum of second inequality in\,\eqref{LB-N-1} is the unique solution of the following equation with boundary condition,
		\begin{equation*}
			\begin{cases}
				\Delta_g\varphi_0(x)=0,\\
				\\
				\varphi_0|{\partial B_{Rr^n_1}(0)}=a^n_1,\ \varphi_0|{\partial B_{\delta}(0)}=b^n_1.
			\end{cases}
		\end{equation*}
		\par
		
		More precisely,
		\begin{equation}\label{LB-N-2*}
			\varphi_0(r)=-{\frac{a^n_1-b^n_1}{\log\delta-\log(Rr^n_1)}}\log r+{\frac{a^n_1\log\delta-b^n_1\log(Rr^n_1)}{\log\delta-\log(Rr^n_1)}}=-{\frac{a^n_1\log{\frac{r}{\delta}}-b^n_1\log{\frac{r}{(Rr^n_1)}}}{\log\delta-\log(Rr^n_1)}}.
		\end{equation}
		\par
		
		In addition, by exploiting \eqref{LB-N-0},\,\eqref{LB-N-1},\,\eqref{LB-N-2*} and applying Taylor expansion, it yields that
		\begin{equation}\label{LB-N-2}
			\begin{aligned}
				&\int_{B_{\delta}(0)\backslash B_{Rr^n_1}(0)}|\nabla_{\mathbb{R}^2}\varphi_0|^2dx\\
				=&{\frac{2\pi(a^n_1-b^n_1)^2}{\log\delta-\log(Rr^n_1)}}\\
				=&{\frac{2\pi\Big(-2(1+\alpha_1(x^0_1))\log r^n_1-\overline{v}^n_1+w_0(x^0_1)+\theta^n_1\Big)^2}{-\log r^n_1}}\cdot\big(1-{\frac{\log R-\log\delta}{-\log r^n_1}}\big)^{-1}\\
				\geq&{\frac{2\pi\Big(-2(1+\alpha_1(x^0_1))\log r^n_1-\overline{v}^n_1+w_0(x^0_1)+\theta^n_1\Big)}{-\log r^n_1}}\cdot\big(1+{\frac{\log R-\log\delta}{-\log r^n_1}}+{\frac{\big(\log{\frac{R}{\delta}}\big)^2}{(\log r^n_1)^2}}\big)\\
				=&-{\frac{(\overline{\rho}_1(x^0_1))^2}{2\pi}}\log r^n_1-2\pi{\frac{(\overline{v}^n_1)^2}{\log r^n_1}}-2\overline{\rho}_1(x^0_1)\overline{v}^n_1+2\overline{\rho}_1(x^0_1)\theta^n_1+2\overline{\rho}_1(x^0_1)w_0(x^0_1)+{\frac{(\overline{\rho}_1(x^0_1))^2}{2\pi}}\log{\frac{R}{\delta}}\\
				&+2\pi{(\frac{\overline{v}^n_1}{\log r^n_1})^2}\log{\frac{R}{\delta}}+{\frac{\overline{v}^n_1}{\log r^n_1}}\big(4\pi\theta^n_1+4\pi w_0(x^0_1)+2\overline{\rho}_1(x^0_1)\log{\frac{R}{\delta}}\big)+o_{\delta}(1)+o_n(1).
			\end{aligned}
		\end{equation}
		\par
		
		In the meanwhile, substituting ${\frac{1}{2}}\big(v^n_1+w_0\big)={\frac{1}{2}}\overline{v}^n_1+\overline{v}^n_2-v^n_2(x)+2w_0(x)$ into $\int_{N_{\delta,n}}Q(v^n_1,v^n_2)dV_g$, we can acquire that
		\begin{equation}\label{LB-N-3}
			\begin{aligned}
				\int_{N_{\delta,n}}Q(v^n_1,v^n_2)dV_g&={\frac{1}{4}}\int_{N_{\delta,n}}|\nabla_g(v^n_1+w_0)|^2dV_g+\int_{N_{\delta,n}}|\nabla_gw_0|^2dV_g-{\frac{1}{2}}\int_{N_{\delta,n}}\nabla_g(v^n_1+w_0)\nabla_gw_0dV_g\\
				&={\frac{1}{4}}\int_{N_{\delta,n}}|\nabla_g(v^n_1+w_0)|^2dV_g+o_n(1).
			\end{aligned}
		\end{equation}
		\par

		Combining \eqref{LB-O},\,\eqref{LB-B},\,\eqref{LB-N-2},\,\eqref{LB-N-3} together and by the upper boundedness of $J_{\rho^n_1}$, it indicates that $\lim\limits_{n\rightarrow+\infty}{\frac{\overline{v}^n_1}{\log r^n_1}}:=2(1+\alpha_1(x^0_1))$. After substituting $\lim\limits_{n\rightarrow+\infty}{\frac{\overline{v}^n_1}{\log r^n_1}}:=2(1+\alpha_1(x^0_1))$ and \eqref{LB-N-1*} into\,\eqref{LB-N-2}, we can obtain that
		\begin{equation}\label{LB-N-4}
			\begin{aligned}
				\int_{N_{\delta,n}}Q(v^n_1,\,v^n_2)dV_g\geq&-{\frac{(\overline{\rho}_1(x^0_1))^2}{2\pi}}\log r^n_1+\overline{\rho}_1(x^0_1)w_0(x^0_1)+\overline{\rho}_1(x^0_1)\theta^n_1+{\frac{(\overline{\rho}_1(x^0_1))^2}{2\pi}}\log{\frac{R}{\delta}}+o_n(1)\\
				=&-{\frac{(\overline{\rho}_1(x^0_1))^2}{2\pi}}\log r^n_1+\overline{\rho}_1(x^0_1)w_0(x^0_1)-2\overline{\rho}_1(x^0_1)\log{\frac{\pi \tilde{h}_1(x^0_1)}{1+\alpha_1(x^0_1)}}-{\frac{(\overline{\rho}_1(x^0_1))^2}{4\pi}}A(x^0_1)\\
				&\quad-{\frac{(\overline{\rho}_1(x^0_1))^2}{2\pi}}\log{\frac{R}{\delta}}+o_n(1).
			\end{aligned}
		\end{equation}
		\par
		
		Based on Remark \ref{AF}, $w_0$ is unique. Then combining \eqref{LB-O},\,\eqref{LB-B},\,\eqref{LB-N-4} and $\lim\limits_{n\rightarrow+\infty}{\frac{\overline{v}^n_1}{\log r^n_1}}:=2(1+\alpha_1(x^0_1))$ together, we can achieve that
		\begin{equation}\label{LB-N-5}
			\begin{aligned}
				&J_{\rho^n_1}(v^n_1,v^n_2)\\
				=&\int_{\Sigma}Q(v^n_1,v^n_2)dV_g+\rho^n_1(x^0_1)\overline{v}^n_1-\rho_2\log\int_{\Sigma}h_2e^{-H_2}e^{v^n_2-\overline{v}^n_2-2w_n}e^{2w_n}dV_g\\
				\geq&\int_{\Sigma}|\nabla_gw_0|^2dV_g-\overline{\rho}_1(x^0_1)\int_{\Sigma}\nabla_gG_{x^0_1}\nabla_gw_0dV_g-\rho_2\log\int_{\Sigma}h_2e^{-4\pi\alpha_2G_0-\overline{\rho}_1(x^0_1)G_{x^0_1}}e^{2w_0}dV_g\\
				&\quad-{\frac{(\overline{\rho}_1(x^0_1))^2}{8\pi}}A(x^0_1)+\overline{\rho}_1(x^0_1)w_0(x^0_1)-\overline{\rho}_1(x^0_1)\log{\frac{{\tilde{h}_1(x^0_1)}}{1+\alpha_1(x^0_1)}}-\overline{\rho}_1(x^0_1)\log\pi-\overline{\rho}_1(x^0_1)\\
				&\quad\quad+o_{\delta}(1)+o_R(1)+o_n(1).
			\end{aligned}
		\end{equation}
		\par

On the other hand, since $(v^n_1,v^n_2)$ is the minimizing sequence of $J_{\rho^n_1}$, then
		\begin{equation}\label{LB-N-6}
		 \inf\limits_{H^1(\Sigma)}J_{\overline{\rho}_1}=\lim\limits_{\delta\rightarrow0}\lim\limits_{R\rightarrow+\infty}\lim\limits_{n\rightarrow+\infty}J_{\rho^n_1}(v^n_1,v^n_2).
		 \end{equation}
		 \par

		By combining Lemma \ref{CPoint},\,\eqref{LB-N-5},\,\eqref{LB-N-6} together and taking limit on the both sides of \eqref{LB-N-5}, we can complete the proof of Lemma \ref{LB}.
	\end{proof}
	\par

	\section{Proof of the main results}
	\subsection{Test function for the case $\alpha_1<0$}
	For the case $\alpha_1<0$, we can choose ${\Phi}^1_{\epsilon}(x)=\Phi_{\epsilon}-w_0$ and ${\Phi}^2_{\epsilon}(x)=-{\frac{1}{2}}\Phi_{\epsilon}+2w_0$ as our testing functions. Similarly to the calculation in \cite{Zhu-a-s}, the function $\Phi_{\epsilon}$ can be defined in the following way:
	\begin{equation*}
		\Phi_{\epsilon}=
		\begin{cases}
			-2\log\big(r^{2(1+\overline{\alpha}_1)}+\epsilon\big)+\log\epsilon,\ ~x\in B_{r_{\epsilon}}(0),\\
			\\
			2\overline{\rho}_1\big(G_{0}(r,\theta)-{\frac{\eta}{8\pi}}\sigma(r,\theta)\big)-2\log\big(\frac{\alpha^{2(1+\overline{\alpha}_1)}_{\epsilon}+1}{\alpha^{2(1+\overline{\alpha}_1)}_{\epsilon}}\big)-{\frac{\overline{\rho}_1}{4\pi}}A(0)+\log\epsilon,\ ~x\in B_{2r_{\epsilon}}(0)\backslash B_{r_{\epsilon}}(0),\\
			\\
			2\overline{\rho}_1G_{0}(r,\theta)-2\log\big(\frac{\alpha^{2(1+\overline{\alpha}_1)}_{\epsilon}+1}{\alpha^{2(1+\overline{\alpha}_1)}_{\epsilon}}\big)-{\frac{\overline{\rho}_1}{4\pi}}A(0)+\log\epsilon,\ \ \ \ \ x\in\Sigma\backslash B_{2r_{\epsilon}}(0),
		\end{cases}
	\end{equation*}
	where $G_{0}(r,\theta)={\frac{1}{8\pi}}\big(-4\log r+A(0)+\sigma(r,\theta)\big),\,r=\rm{dist}_g(x,0)$ and $\eta(x)\in C^{\infty}_0(B_{2r_{\epsilon}}(0))$ is a cut-off function satisfying $\eta\equiv1$ in $B_{r_{\epsilon}}(0),\ \|\nabla\eta\|\leq Cr^{-1}_{\epsilon}$. In addition, $\alpha_{\epsilon}:={\frac{\epsilon^{-\frac{1}{2(1+\overline{\alpha}_1)}}}{-\log\epsilon}}\rightarrow+\infty,\,r_{\epsilon}:=\alpha_{\epsilon}\epsilon^{\frac{1}{2(1+\overline{\alpha}_1)}}\rightarrow0$ as $\epsilon\rightarrow0.$ 
	\par

	By straightforward computation, we immediately attain
	\ $ \int_{\Sigma}Q(\Phi^1_{\epsilon},\Phi^2_{\epsilon})dV_g={\frac{1}{4}}\int_{\Sigma}|\nabla_g\Phi_{\epsilon}|^2dV_g+\int_{\Sigma}|\nabla_gw_0|^2dV_g-{\frac{1}{2}}\int_{\Sigma}\nabla_g\Phi_{\epsilon}\nabla_gw_0dV_g$ and $\int_{\Sigma}\Phi^1_{\epsilon}dV_g=\int_{\Sigma}\Phi_{\epsilon}dV_g,\ \int_{\Sigma}\Phi^2_{\epsilon}dV_g=-{\frac{1}{2}}\int_{\Sigma}\Phi_{\epsilon}dV_g$.
	In the meanwhile, we can also arrive at
	\begin{equation}\label{TFP-1}
		\begin{aligned}
			{\frac{1}{4}}\int_{\Sigma}|\nabla_g\Phi_{\epsilon}|^2dV_g&=-\overline{\rho}_1\log\epsilon-\overline{\rho}_1+{\frac{(\overline{\rho}_1)^2}{8\pi}} A(0)+o_{\epsilon}(1),
		\end{aligned}
	\end{equation}

	\begin{equation}\label{TFP-2}
		\begin{aligned}
			-\overline{\rho}_1\log\int_{\Sigma}h_1e^{-H_1}e^{\Phi_{\epsilon}-w_0}dV_g=&-\overline{\rho}_1\log\big({\frac{\pi h_1(0)e^{-\frac{\overline{\alpha}_1}{2}A(0)}}{1+\overline{\alpha}_1}}\big)+\overline{\rho}_1w_0(0)+o_{\epsilon}(1),
		\end{aligned}
	\end{equation}
	
	\begin{equation}\label{TFP-3}
		\begin{aligned}
			\big(\overline{\rho}_1-{\frac{\rho_2}{2}}\big)\int_{\Sigma}\Phi_{\epsilon}dV_g=\big(\overline{\rho}_1-{\frac{\rho_2}{2}}\big)\log\epsilon-\big(\overline{\rho}_1-{\frac{\rho_2}{2}}\big)\big(1+\overline{\alpha}_1\big)A(0)+o_{\epsilon}(1),
		\end{aligned}
	\end{equation}
	
	\begin{equation}\label{TFP-4}
		-{\frac{1}{2}}\int_{\Sigma}\nabla_g\Phi_{\epsilon}\nabla_gw_0dV_g=-\overline{\rho}_1\int_{\Sigma}\nabla_gG_{0}\nabla_gw_0dV_g+o_{\epsilon}(1),
	\end{equation}
	
	\begin{equation}\label{TFP-5}
		\begin{aligned}
			-\rho_2\log\int_{\Sigma}h_2e^{-H_2}e^{-{\frac{1}{2}}\Phi_{\epsilon}+2w_0}dV_g=&{\frac{\rho_2}{2}}\log\epsilon-\rho_2\log\int_{\Sigma}h_2e^{-4\pi\alpha_2G_0-\overline{\rho}_1G_{0}}e^{2w_0}dV_g\\
			&\quad-{\frac{\rho_2}{2}}(1+\overline{\alpha}_1)A(0)+o_{\epsilon}(1).
		\end{aligned}
	\end{equation}
	\par

	By combining\eqref{TFP-1},\,\eqref{TFP-2},\,\eqref{TFP-3},\,\eqref{TFP-4} and\,\eqref{TFP-5} together, we can conclude that
	\begin{equation}\label{TFP-6}
		\begin{aligned}
			J_{\overline{\rho}_1}(\Phi^1_{\epsilon},\Phi^2_{\epsilon})=&\int_{\Sigma}Q(\Phi^1_{\epsilon},\Phi^2_{\epsilon})dV_g+\overline{\rho}_1\int_{\Sigma}\Phi^1_{\epsilon}dV_g+\rho_2\int_{\Sigma}\Phi^2_{\epsilon}dV_g-\overline{\rho}_1\log\int_{\Sigma}h_1e^{-H_1}e^{\Phi^1_{\epsilon}}dV_g\\
			&\quad-\rho_2\log\int_{\Sigma}h_2e^{-H_2}e^{\Phi^2_{\epsilon}}dV_g+o_{\epsilon}(1)\\
			=&{\frac{1}{4}}\int_{\Sigma}|\nabla_g\Phi_{\epsilon}|^2dV_g+\int_{\Sigma}|\nabla_gw_0|^2dV_g-{\frac{1}{2}}\int_{\Sigma}\nabla_g\Phi_{\epsilon}\nabla_gw_0dV_g+\big(\overline{\rho}_1-{\frac{1}{2}}\rho_2\big)\int_{\Sigma}\Phi_{\epsilon}\,dV_g\\
			&\quad-\overline{\rho}_1\log\int_{\Sigma}h_1e^{-H_1}e^{\Phi_{\epsilon}-w_0}dV_g-\rho_2\log\int_{\Sigma}h_2e^{-H_2}e^{-{\frac{1}{2}}\Phi_{\epsilon}+2w_0}dV_g+o_{\epsilon}(1)\\
			=&\int_{\Sigma}|\nabla_gw_0|^2dV_g-\overline{\rho}_1\int_{\Sigma}\nabla_gG_0\nabla_gw_0dV_g-\rho_2\log\int_{\Sigma}h_2e^{-4\pi\alpha_2G_0-\overline{\rho}_1G_0}e^{2w_0}dV_g\\
			&\quad-{\frac{(\overline{\rho}_1)^2}{8\pi}} A(0)+\overline{\rho}_1w_0(0)-\overline{\rho}_1\log\big({\frac{h_1(0)e^{-\frac{\overline{\alpha}_1}{2}A(0)}}{1+\overline{\alpha}_1}}\big)-\overline{\rho}_1\log\pi-\overline{\rho}_1+o_{\epsilon}(1).
		\end{aligned}
	\end{equation}
	\par

Thus, by taking the limit on both sides of \eqref{TFP-6}, we can conclude that
	\begin{equation}\label{TFP-7}
		\begin{aligned}
			\lim\limits_{\epsilon\rightarrow0}J_{\overline{\rho}_1}(\Phi^1_{\epsilon},\Phi^2_{\epsilon})=&\int_{\Sigma}|\nabla_gw_0|^2dV_g-\overline{\rho}_1\int_{\Sigma}\nabla_gG_0\nabla_gw_0dV_g-\rho_2\log\int_{\Sigma}h_2e^{-4\pi\alpha_2G_0-\overline{\rho}_1G_0}e^{2w_0}dV_g\\
			&\quad-{\frac{(\overline{\rho}_1)^2}{8\pi}}A(0)+\overline{\rho}_1w_0(0)-\overline{\rho}_1\log\big({\frac{{h_1(0)e^{-\frac{\overline{\alpha}_1}{2}A(0)}}}{1+\overline{\alpha}_1}}\big)-\overline{\rho}_1\log\pi-\overline{\rho}_1.
		\end{aligned}
	\end{equation}
	\par

\subsection{Test function for the case $\alpha_1\geq0$}
	In this section, when $\alpha_1\geq0$, with the goal of deducing the similar geometric condition as presented in \cite{JLW}, we can provide a more precise quantification of the error terms, analogous to the formula in \cite{DJLW}. In order to accomplish this goal, it is indispensable to derive a more accurate expression of the Green function and to make slight modifications of the testing function $\Phi_{\epsilon}$. For any $x_0\in\Sigma_+:=\{x\in\Sigma:\ h_1(x)>0\}$ as $\alpha_1=0$ and $x_0\in\Sigma_+\backslash\{0\}:=\{x\in\Sigma:\ h_1(x)>0\}$ as $\alpha_1>0$, define
	\begin{equation*}
		\Phi_{\epsilon}=
		\begin{cases}
			-2\log\big(r^2+\epsilon\big)+b_1r\cos\theta+b_2r\sin\theta+\log\epsilon,\ ~x\in B_{r_{\epsilon}}(x_0),\\
			\\
			8\pi\big(G_{x_0}-{\frac{\eta}{8\pi}}\beta(r,\theta)\big)-2\log\big(\frac{\alpha^{2}_{\epsilon}+1}{\alpha^{2}_{\epsilon}}\big)-A(x_0)+\log\epsilon,\ ~x\in B_{2r_{\epsilon}}(x_0)\backslash B_{r_{\epsilon}}(x_0),\\
			\\
			8\pi G_{x_0}-2\log\big(\frac{\alpha^{2}_{\epsilon}+1}{\alpha^{2}_{\epsilon}}\big)-A(x_0)+\log\epsilon,\ \ \ \ \ x\in\Sigma\backslash B_{2r_{\epsilon}}(x_0),
		\end{cases}
	\end{equation*}
	where $G_{x_0}={\frac{1}{8\pi}}\big(-4\log r+A(x_0)+b_1r\cos\theta+b_2r\sin\theta+\beta(r,\theta)\big),\,r_{\epsilon}:=\alpha_{\epsilon}\sqrt{\epsilon}$\ and\ $\eta(r)\in C^{\infty}_0(B_{2r_{\epsilon}}(x_0))$ is a cut-off function such that $\eta\equiv1$ in $B_{r_{\epsilon}}(x_0),\ |\nabla\eta|\leq Cr^{-1}_{\epsilon}$. In addition, we can take $\alpha^4_{\epsilon}\epsilon:={\frac{1}{\log(-\log\epsilon)}}$ such that $\alpha_{\epsilon}\rightarrow+\infty$ and $r_{\epsilon}\rightarrow0$ as $\epsilon\rightarrow0$.
	
	By straightforward calculation, it yields that
	\begin{equation}\label{AP-1}
		\begin{aligned}
			{\frac{1}{4}}\int_{\Sigma}|\nabla_g\Phi_{\epsilon}|^2dV_g=&-4\pi\log\epsilon+4\pi\log({\frac{\alpha^2_{\epsilon}+1}{\alpha^2_{\epsilon}}})-4\pi{\frac{\alpha^2_{\epsilon}}{\alpha^2_{\epsilon}+1}}+2\pi A(x_0)\\
			&\quad+{\frac{8\pi}{6}}K(x_0)\epsilon\cdot\log(\alpha^2_{\epsilon}+O\big(\alpha^4_{\epsilon}\epsilon^2\log(\alpha^2_{\epsilon}\epsilon)\big),
		\end{aligned}
	\end{equation}
	\par

	\begin{equation}\label{AP-2}
		\begin{aligned}
			(4\pi-{\frac{\rho_2}{2}})\int_{\Sigma}\Phi_{\epsilon}dV_g&=(4\pi-{\frac{\rho_2}{2}})\log\epsilon-(8\pi-\rho_2)\big(1-|B_{\alpha_{\epsilon}\sqrt{\epsilon}}|\big)\log({\frac{\alpha^2_{\epsilon}+1}{\alpha^2_{\epsilon}}})\\
			&\quad-(8\pi-\rho_2)\pi\alpha^2_{\epsilon}\epsilon\log\big({\frac{\alpha^2_{\epsilon}+1}{\alpha^2_{\epsilon}}}\big)-(4\pi-{\frac{\rho_2}{2}})A(x_0)\\
			&\quad-(8\pi-\rho_2)\pi\epsilon\log(\alpha^2_{\epsilon}+1)+O\big(\alpha^4_{\epsilon}\epsilon^2\log(\alpha^2_{\epsilon}\epsilon)\big),
		\end{aligned}
	\end{equation}

	\begin{equation}\label{AP-3}
		-{\frac{1}{2}}\int_{\Sigma}\nabla_g\Phi_{\epsilon}\nabla_gw_0dV_g=-4\pi\int_{\Sigma}\nabla_gG_{x_0}\nabla_gw_0dV_g-\pi\rho_2\epsilon\log(\alpha^2_{\epsilon}+1)+O(\epsilon),
	\end{equation}
	
	\begin{equation}\label{AP-4}
		\begin{aligned}
			-\rho_2\log\int_{\Sigma}h_2e^{-H_2}e^{-{\frac{1}{2}}\Phi_{\epsilon}+2w_0}dV_g=&-\rho_2\log(\frac{\alpha^2_{\epsilon}+1}{\alpha^2_{\epsilon}})-\rho_2\log\int_{\Sigma}h_2e^{-4\pi\alpha_2G_0-4\pi G_{x_0}}e^{2w_0}dV_g\\
			&\quad+{\frac{\rho_2}{2}}\log\epsilon-{\frac{\rho_2}{2}}A(x_0)+O(\epsilon).
		\end{aligned}
	\end{equation}
	\par
	
	Relying on Taylor expansion $	h_1(x)e^{-4\pi\alpha_1G_0(x)}e^{-w_0(x)}-h_1(x_0)e^{-4\pi\alpha_1G_0(x_0)}e^{-w_0(x_0)}=k_1r\cos\theta+k_2r\sin\theta+k_3r^2\cos^2\theta+2k_4r^2\cos\theta\sin\theta+k_5r^2\sin^2\theta+O(r^3)$, we also have that
	\begin{equation}\label{AP-5}
		\begin{aligned}
			&-4\pi\log\int_{\Sigma}h_1e^{-4\pi\alpha_1G_0}e^{\Phi_{\epsilon}-w_0}dV_g\\
			=&-4\pi\log\big(h_1(x_0)e^{-4\pi\alpha_1G_0(x_0)}\big)+4\pi w_0(x_0)-4\pi\log\pi-4\pi\log({\frac{\alpha^2_{\epsilon}}{\alpha^2_{\epsilon}+1}})-{\frac{4\pi}{\alpha^2_{\epsilon}+1}}\\
			&\quad-{\frac{1}{4\big(\alpha^2_{\epsilon}+1\big)^2}}+{\frac{2\pi K(x_0)}{3}}\epsilon\log(\alpha^2_{\epsilon}+1)-\pi\big(b^2_1+b^2_2\big)\epsilon\log(\alpha^2_{\epsilon}+1)\\
			&\quad+\pi\big(b^2_1+b^2_2\big)\epsilon\log(\alpha^2_{\epsilon}\epsilon)-\pi{\frac{\Delta_g(h_1e^{-4\pi\alpha_1G_0}e^{-w_0})|_{x=x_0}}{h_1(x_0)e^{-4\pi\alpha_1G_0(x_0)}e^{-w_0(x_0)}}}\epsilon\log(\alpha^2_{\epsilon}+1)\\
			&\quad+\pi{\frac{\Delta_g(h_1e^{-4\pi\alpha_1G_0}e^{-w_0})|_{x=x_0}}{h_1(x_0)e^{-4\pi\alpha_1G_0(x_0)}e^{-w_0(x_0)}}}\epsilon\log(\alpha^2_{\epsilon}\epsilon)-2\pi{\frac{k_{1}b_1+k_{2}b_2}{h_ 1(x_0)e^{-4\pi\alpha_1G_0(x_0)}e^{-w_0(x_0)}}}\epsilon\log(\alpha^2_{\epsilon}+1)\\
			&\quad+2\pi\big(c_1+c_3-{\frac{K(x_0)}{3}}\big)\epsilon\log(\alpha^2_{\epsilon}\epsilon)+2\pi{\frac{k_{1}b_1+k_{2}b_2}{h_1(x_0)e^{-4\pi\alpha_1G_0(x_0)}e^{-w_0(x_0)}}}\epsilon\log(\alpha^2_{\epsilon}\epsilon)+O(\epsilon),
		\end{aligned}
	\end{equation}
	\par
	
	Combining\eqref{AP-1},\,\eqref{AP-2},\,\eqref{AP-3},\,\eqref{AP-4} and\,\eqref{AP-5} together, we can obtain that
	\begin{equation*}
		\begin{aligned}
			&J_{\overline{\rho}_1}(\Phi^1_{\epsilon},\Phi^2_{\epsilon})\\
			=&\int_{\Sigma}Q(\Phi^1_{\epsilon},\Phi^2_{\epsilon})dV_g+4\pi\int_{\Sigma}\Phi^1_{\epsilon}dV_g+\rho_2\int_{\Sigma}\Phi^2_{\epsilon}dV_g-4\pi\log\int_{\Sigma}h_1e^{-4\pi\alpha_1G_0}e^{\Phi^1_{\epsilon}}dV_g\\
			&\quad-\rho_2\log\int_{\Sigma}h_2e^{-4\pi\alpha_2G_0}e^{\Phi^2_{\epsilon}}dV_g\\
			=&{\frac{1}{4}}\int_{\Sigma}|\nabla_g\Phi_{\epsilon}|^2dV_g+\int_{\Sigma}|\nabla_gw_0|^2dV_g-{\frac{1}{2}}\int_{\Sigma}\nabla_g\Phi_{\epsilon}\nabla_gw_0dV_g+\big(4\pi-{\frac{1}{2}}\rho_2\big)\int_{\Sigma}\Phi_{\epsilon}dV_g\\
			&\quad-4\pi\log\int_{\Sigma}h_1e^{-4\pi\alpha_1G_0}e^{\Phi_{\epsilon}-w_0}dV_g-\rho_2\log\int_{\Sigma}h_2e^{-4\pi\alpha_2G_0}e^{-{\frac{1}{2}}\Phi_{\epsilon}+2w_0}dV_g\\
		\end{aligned}
	\end{equation*}

	\begin{equation}\label{AP-6}
		\begin{aligned}	
			=&\int_{\Sigma}|\nabla_gw_0|^2dV_g-4\pi\int_{\Sigma}\nabla_gG_{x_0}\nabla_gw_0dV_g-\rho_2\log\int_{\Sigma}h_2e^{-4\pi\alpha_2G_0}e^{-{4\pi}G_{x_0}}e^{2w_0}dV_g\\
			&-2\pi A(x_0)+4\pi w_0(x_0)-4\pi\log\big( h_1(x_0)e^{-4\pi\alpha_1G_0(x_0)}\big)-4\pi\log\pi-4\pi\\
			&\quad-8\pi^2\Big(1-{\frac{1}{4\pi}}K(x_0)+{\frac{b^2_1+b^2_2}{8\pi}}+{\frac{\Delta_g(h_1e^{-4\pi\alpha_1G_0}e^{-w_0})|_{x=x_0}}{8\pi h_1(x_0)e^{-4\pi\alpha_1G_0(x_0)}e^{-w_0(x_0)}}}\\
			&\quad\quad+{\frac{k_1b_1+k_2b_2}{4\pi h_1(x_0)e^{-4\pi\alpha_1G_0(x_0)}e^{-w_0(x_0)}}}\Big)\epsilon\log\big(\alpha^2_{\epsilon}+1\big)+O\big(\epsilon\log(\alpha^2_{\epsilon}\epsilon)\big).
		\end{aligned}
	\end{equation}
	\par
	
	In the meanwhile, by straightforward computation, we can have that
	\begin{equation}\label{AP-7}
		\begin{aligned} 
	  \Delta_g(h_1e^{-4\pi\alpha_1G_0}e^{-w_0})\big|_{x=x_0}=&e^{-4\pi\alpha_1G_0(x_0)-w_0(x_0)}\Big(\Delta_gh_1-2(4\pi\alpha_1+1)\nabla_gh_1\nabla_gw_0\\
	  &\quad-4\pi\alpha_1h_1\Delta_gG_0+8\pi\alpha_1h_1\nabla_gG_0\nabla_gw_0+(4\pi\alpha_1)^2h_1|\nabla_gG_0|^2\\
	  &\quad+h_1|\nabla_gw_0|^2-h_1\Delta_gw_0\Big)|_{x=x_0},
	  \end{aligned}
	  \end{equation}
	  \par

	  \begin{equation}\label{AP-8}
	  	\begin{aligned}
	    k^2_1+k^2_2=&e^{-8\pi\alpha_1G_0(x_0)-2w_0(x_0)}\Big(|\nabla_gh_1(x_0)|^2+(4\pi\alpha_1)^2h^2_1|\nabla_gG_0|^2+h^2_1(x_0)|\nabla_gw_0(x_0)|^2\\
	    &\quad-8\pi\alpha_1h_1(x_0)\nabla_gh_1(x_0)\nabla_gG_0(x_0)-2h_1(x_0)\nabla_gh_1(x_0)\nabla_gw_0(x_0)\\
	    &\quad\quad+8\pi\alpha_1h^2_1(x_0)\nabla_gG_0(x_0)\nabla_gw_0(x_0)\Big),
	    \end{aligned}
	    \end{equation}
	    \par

	    \begin{equation}\label{AP-9}
	    	\begin{aligned}
	    \Delta_g\log  h_1(x)|_{x=x_0}=-{\frac{|\nabla_gh_1(x_0)|^2}{h^2_1(x_0)}}+{\frac{\Delta_gh_1(x_0)}{h_1(x_0)}},
	        \end{aligned}
	    \end{equation}
	    \par

	   	\begin{equation}\label{AP-10}
	   	\begin{aligned}
	   		\Delta_gw_0(x_0)=\rho_2.
	   	\end{aligned}
	   \end{equation}
	   \par

Then combining \eqref{AP-6},\eqref{AP-7},\eqref{AP-8},\eqref{AP-9},\eqref{AP-10} together and following from the assumption that $\Delta_g\log h_1(x)+8\pi-4\pi\alpha_1-\rho_2>2K(x)$ on $\Sigma_+$, we can calculate the coefficient of the term $\epsilon\log(\alpha^2_{\epsilon}+1)$,
	\begin{equation*}
		\begin{aligned}
			&1-{\frac{1}{4\pi}}K(x_0)+{\frac{b^2_1+b^2_2}{8\pi}}+{\frac{\Delta_g(h_1e^{-4\pi\alpha_1G_0}e^{-w_0})|_{x=x_0}}{8\pi h_1(x_0)e^{-4\pi\alpha_1G_0(x_0)}e^{-w_0(x_0)}}}+{\frac{k_1b_1+k_2b_2}{4\pi h_1(x_0)e^{-4\pi\alpha_1G_0(x_0)}e^{-w_0(x_0)}}}\\
			=&1-{\frac{1}{4\pi}}K(x_0)-{\frac{\Delta_g w_0(x_0)}{8\pi}}+{\frac{b^2_1+b^2_2}{8\pi}}+{\frac{\Delta_gh_1(x_0)}{8\pi h_1(x_0)}}-{\frac{(4\pi\alpha_1+1)\nabla_gh_1(x_0)\nabla_gw_0(x_0)}{4\pi h_1(x_0)}}\\
			&+{\frac{8\pi\alpha_1\nabla_gG_0(x_0)\nabla_gw_0(x_0)}{8\pi}}+{\frac{|\nabla_g w_0(x_0)|^2}{8\pi}}+{\frac{(4\pi\alpha_1)^2|\nabla_gG_0(x_0)|^2}{8\pi}}\\
			&\quad-{\frac{4\pi\alpha_1}{8\pi}}+{\frac{k_1b_1+k_2b_2}{4\pi h_1(x_0)e^{-4\pi\alpha_1G_0(x_0)}e^{-w_0(x_0)}}}
		\end{aligned}
	\end{equation*}
	\par

	\begin{equation*}
		\begin{aligned}
			=&{\frac{\Delta_g\log h_1(x_0)}{8\pi}}+1-{\frac{1}{4\pi}}K(x_0)-{\frac{\rho_2}{8\pi}}+{\frac{b^2_1+b^2_2}{8\pi}}+{\frac{|\nabla_gh_1(x_0)|^2}{8\pi h^2_1(x_0)}}-{\frac{(4\pi\alpha_1+1)\nabla_gh_1(x_0)\nabla_gw_0(x_0)}{4\pi h_1(x_0)}}\\
			&\quad+{\frac{8\pi\alpha_1\nabla_gG_0(x_0)\nabla_gw_0(x_0)}{8\pi}}+{\frac{|\nabla w_0(x_0)|^2}{8\pi}}+{\frac{(4\pi\alpha_1)^2|\nabla_gG_0(x_0)|^2}{8\pi}}\\
			&\quad\quad-{\frac{4\pi\alpha_1}{8\pi}}+{\frac{k_1b_1+k_2b_2}{4\pi h_1(x_0)e^{-4\pi\alpha_1G_0(x_0)}e^{-w_0(x_0)}}}\\
			=&{\frac{1}{8\pi}}\Big(\Delta_g\log h_1(x_0)+8\pi-2K(x_0)-4\pi\alpha_1-\rho_2+b^2_1+b^2_2+|\nabla_gw_0(x_0)|^2\\
			&\quad-{\frac{(8\pi\alpha_1+2)\nabla_gh_1(x_0)\nabla_gw_0(x_0)}{h_1(x_0)}}+{\frac{2\big(k_1b_1+k_2b_2\big)}{h_1(x_0)e^{-4\pi\alpha_1G_0(x_0)}e^{-w_0(x_0)}}}\\
			&\quad+{\frac{|\nabla_gh_1(x_0)|^2}{h^2_1(x_0)}}+8\pi\alpha_1\nabla_gG_0(x_0)\nabla_gw_0(x_0)+(4\pi\alpha_1)^2|\nabla_gG_0(x_0)|^2
			\Big)\\\
			=&{\frac{1}{8\pi}}\Big(\Delta_g\log h_1(x_0)+8\pi-4\pi\alpha_1-\rho_2-2K(x_0)+\big(b_1+{\frac{k_1}{h_1(x_0)e^{-4\pi\alpha_1G_0(x_0)}e^{-w_0(x_0)}}}\big)^2\\
			&\quad+\big(b_2+{\frac{k_2}{h_1(x_0)e^{-4\pi\alpha_1G_0(x_0)}e^{-w_0(x_0)}}}\big)^2\Big)>0.
		\end{aligned}
	\end{equation*}
	\par

	If $\alpha_1>0$, by choosing special $x_0\in\Sigma_+\backslash\{0\}$, we can derive that
	\begin{equation}\label{AP-11}
		\begin{aligned}
			J_{\overline{\rho}_1}(\Phi^1_{\epsilon},\Phi^2_{\epsilon})<&\int_{\Sigma}|\nabla_gw_0|^2dV_g-4\pi\int_{\Sigma}\nabla_gG_{x_0}\nabla_gw_0dV_g-\rho_2\log\int_{\Sigma}h_2e^{-4\pi\alpha_2G_0}e^{-{4\pi}G_{x_0}}e^{2w_0}dV_g\\
			&\quad-2\pi A(x_0)+4\pi w_0(x_0)-4\pi\log\big(h_1(x_0)e^{-4\pi\alpha_1G_0(x_0)}\big)-4\pi\log\pi-4\pi\\
			=&\min\limits_{x\in\Sigma_+\backslash\{0\}}\Big(\int_{\Sigma}|\nabla_gw_0|^2dV_g-4\pi\int_{\Sigma}\nabla_gG_{x}\nabla_gw_0dV_g-\rho_2\log\int_{\Sigma}h_2e^{-4\pi\alpha_2G_0}e^{-{4\pi}G_{x}}e^{2w_0}dV_g\\
			&\quad-2\pi A(x)+4\pi w_0(x)-4\pi\log\big(h_1(x)e^{-4\pi\alpha_1G_0(x)}\big)\Big)-4\pi\log\pi-4\pi.
		\end{aligned}
	\end{equation}
	\par
	
    If $\alpha_1=0$, similarly by choosing special $x_0\in\Sigma_+$, we can also obtain that
	\begin{equation}\label{AP-12}
		\begin{aligned}
			J_{\overline{\rho}_1}(\Phi^1_{\epsilon},\Phi^2_{\epsilon})<&\min\limits_{x\in\Sigma_+}\Big(\int_{\Sigma}|\nabla_gw_0|^2dV_g-4\pi\int_{\Sigma}\nabla_gG_{x}\nabla_gw_0dV_g-\rho_2\log\int_{\Sigma}h_2e^{-4\pi\alpha_2G_0}e^{-{4\pi}G_{x}}e^{2w_0}dV_g\\
			&\quad-2\pi A(x)+4\pi w_0(x)-4\pi\log{h_1(x)}\Big)-4\pi\log\pi-4\pi.
		\end{aligned}
	\end{equation}
	\par

	\noindent\textbf{Completion of the proof for Theorem \ref{Exist-thm} and Theorem \ref{Non-thm}.} It is evident that the conclusion (i) of Theorem \ref{Exist-thm} holds following from \eqref{AP-11},\ \eqref{AP-12} and Lemma \ref{LB};\ the conclusion (ii) of Theorem \ref{Exist-thm} holds following from Lemma \ref{LB}. In the meanwhile, according to Lemma \ref{LB} and \eqref{TFP-7}, we can also prove Theorem \ref{Non-thm}.
	\par

\begin{remark}
		If $\alpha_1<0$, we can't deduce the similar geometric condition as above. In fact, for the case $\alpha_1\geq0$, we can derive the geometric condition by calculating the biggest error term $\epsilon\log(\alpha^2_{\epsilon}+1)$. However, when we consider $\alpha_1<0$, it is difficult to compare these small error terms.
		\par	
		
		More specifically, we can define a more accurate testing function $\Phi_{\epsilon}$ as follows,
		\begin{equation*}
			\Phi_{\epsilon}=
			\begin{cases}
				-2\log\big(r^{2(1+\overline{\alpha}_1)}+\epsilon\big)+{\frac{\overline{\rho}_1}{4\pi}}b_1r\cos\theta+{\frac{\rho_2}{4\pi}}b_2r\sin\theta+\log\epsilon,\ ~x\in B_{r_{\epsilon}}(0),\\
				\\
				2\overline{\rho}_1\big(G_0-{\frac{\eta}{8\pi}}\beta(r,\theta)\big)-2\log\big(\frac{\alpha^{2(1+\overline{\alpha}_1)}_{\epsilon}+1}{\alpha^{2(1+\overline{\alpha}_1)}_{\epsilon}}\big)-{\frac{\overline{\rho}_1}{4\pi}}A(0)+\log\epsilon,\ ~x\in B_{2r_{\epsilon}}(0)\backslash B_{r_{\epsilon}}(0),\\
				\\
				2\overline{\rho}_1G_{0}-2\log\big(\frac{\alpha^{2(1+\overline{\alpha}_1)}_{\epsilon}+1}{\alpha^{2(1+\overline{\alpha}_1)}_{\epsilon}}\big)-{\frac{\overline{\rho}_1}{4\pi}}A(0)+\log\epsilon,\ \ \ \ \ x\in\Sigma\backslash B_{2r_{\epsilon}}(0),
			\end{cases}
		\end{equation*}
		where $G_{x_0}={\frac{1}{8\pi}}\big(-4\log r+A(0)+b_1r\cos\theta+b_2r\sin\theta+\beta(r,\theta)\big),\ r_{\epsilon}:=\alpha_{\epsilon}\epsilon^{\frac{1}{2(1+\overline{\alpha}_1)}},\,r=\rm{dist}_g(x,0)$\,and\ $\eta(x)\in C^{\infty}_0(B_{2r_{\epsilon}}(0))$ is a cut-off function satisfying $\eta\equiv1$ in $B_{r_{\epsilon}}(0),\,\|\nabla\eta\|\leq Cr^{-1}_{\epsilon}$. In addition, we will choose $\alpha_{\epsilon}$ later such that $\alpha_{\epsilon}\rightarrow+\infty\ \text{and}\,r_{\epsilon}\rightarrow0$ as $\epsilon\rightarrow0.$ 
		\par

		By straightforward calculation, we can acquire that
		\begin{equation*}
			\begin{aligned}	
				{\frac{1}{4}}\int_{\Sigma}|\nabla_g\Phi_{\epsilon}|^2dV_g=&-\overline{\rho}_1\log\epsilon+\overline{\rho}_1\log\big(\frac{1+\alpha^{2(1+\overline{\alpha}_1)}_{\epsilon}}{\alpha^{2(1+\overline{\alpha}_1)}_{\epsilon}}\big)-\overline{\rho}_1{\frac{\alpha^{2(1+\overline{\alpha}_1)}}{1+\alpha^{2(1+\overline{\alpha}_1)}_{\epsilon}}}+{\frac{(\overline{\rho}_1)^2}{8\pi}}A(0)\\
				&\quad+{\frac{(\overline{\rho}_1)^2K(0)}{6\pi}}\int^{r_{\epsilon}}_0{\frac{\epsilon r}{r^{2(1+\overline{\alpha}_1)}+\epsilon}}dr+O\big(r^4_{\epsilon}\log( r^2_{\epsilon})\big)+O(\epsilon).
			\end{aligned}
		\end{equation*}
		\par
		
		Unfortunately, the value of ${\frac{(\overline{\rho}_1)^2K(0)}{6\pi}}\int^{r_{\epsilon}}_0{\frac{\epsilon r}{r^{2(1+\overline{\alpha}_1)}+\epsilon}}dr$ varies with the value of $\overline{\alpha}_1$. For example, if $\alpha_1=\overline{\alpha}_1=-{\frac{3}{4}}$, then
		\begin{equation*}
			{\frac{(\overline{\rho}_1)^2K(0)}{6\pi}}\int^{r_{\epsilon}}_0{\frac{\epsilon r}{r^{2(1+\overline{\alpha}_1)}+\epsilon}}dr={\frac{(\overline{\rho}_1)^2K(0)}{6\pi}}\epsilon\int^{r_{\epsilon}}_0{\frac{ r}{r^{\frac{1}{2}}+\epsilon}}dr<{\frac{(\overline{\rho}_1)^2K(0)}{6\pi}}\epsilon\int^{r_{\epsilon}}_0r^{\frac{1}{2}}dr=O(\epsilon).
		\end{equation*}
		\par
		
		This fact implies that $\int_{\Sigma}|\nabla_g\Phi_{\epsilon}|^2dV_g$ can not generate  the bigger error term which contributes to the Gaussian curvature $K(x)$.		
\end{remark}

\end{document}